\documentclass[10pt,a4paper]{article}
\usepackage{amsfonts}
\usepackage[latin9]{inputenc}
\usepackage{amsmath}
\usepackage{amssymb}
\usepackage{breqn}
\usepackage{url}
\usepackage{graphicx}
\usepackage[pdfstartview=FitH]{hyperref}
\usepackage{amsthm}
\usepackage{authblk}
\usepackage{hyperref}
\usepackage{url}
\makeatletter
\begin{document}
\newtheorem{theorem}{Theorem}[section]
\newtheorem{lemma}[theorem]{Lemma}
\newtheorem{definition}[theorem]{Definition}
\newtheorem{claim}[theorem]{Claim}
\newtheorem{example}[theorem]{Example}
\newtheorem{remark}[theorem]{Remark}
\newtheorem{proposition}[theorem]{Proposition}
\newtheorem{corollary}[theorem]{Corollary}
\newtheorem{observation}[theorem]{Observation}

\title{Averaged projections, angles between groups and strengthening of Banach property (T)}
\author{Izhar Oppenheim}
\affil{Department of Mathematics\\
 Ben-Gurion University of the Negev  \\
 Be'er Sheva 84105, Israel \\
E-mail: izharo@bgu.ac.il}

\maketitle
\textbf{Abstract}. Recently, Lafforgue introduced a new strengthening of Banach property (T), which he called strong Banach property (T) and showed that this property has implications regarding fixed point properties and Banach expanders. In this paper, we introduce a new strengthening of Banach property (T), called ``robust Banach property (T)'', which is weaker than strong Banach property (T), but is still strong enough to ensure similar applications.   
Using the method of averaged projections in Banach spaces and introducing a new notion of angles between projections, we establish a criterion for robust Banach property (T) and show several examples of groups in which this criterion is fulfilled. We also derive several applications regarding fixed point properties and Banach expanders and give examples of these applications. 
%\textbf{Mathematics Subject Classification (2010)}. Primary ,Secondary. \\
%\textbf{Keywords}. .

\section{Introduction}
In \cite{Lafforgue1} and \cite{Lafforgue2} V. Lafforgue introduced a very strong variant of property (T), which he named strong Banach property (T) and proved that $SL_3 (\mathbb{F})$, where $\mathbb{F}$ is a non-archimedean local field, has strong Banach property (T).  After Lafforgue's seminal work, his techniques were developed and generalized in \cite{Liao}, \cite{Salle} and \cite{LaatSalle}. We shall start by reviewing the definition of Lafforgue and then introduce a weaker version of this definition which we call robust Banach property (T).

Let $G$ be a locally compact group. Let $\mathcal{F}$ be a family of linear representations on Banach spaces, $\pi : G \rightarrow \mathcal{B} (X)$, that are continuous with respect to the strong operator topology. Define the norm $\Vert . \Vert_{\mathcal{F}}$ on $C_c (G)$ as $\Vert f \Vert_{\mathcal{F}} =  \sup_{\pi \in \mathcal{F}} \Vert \pi (f) \Vert$.
Define $C_{\mathcal{F}} (G)$ to be the completion of $C_c (G)$ with respect to this norm. 
If $\mathcal{F}$ is closed under complex conjugation (i.e., $\pi \in \mathcal{F} \Rightarrow \overline{\pi} \in \mathcal{F}$) and under duality (i.e., $\pi \in \mathcal{F} \Rightarrow \pi^* \in \mathcal{F}$), then $C_{\mathcal{F}} (G)$ is a Banach algebra with an involution 
$$f^* (g) = \overline{f(g^{-1})}, \forall g \in G .$$
%\begin{definition}
%Let $G$ be a locally compact group and let $\mathcal{F}$ be a family of linear representations of $G$ on Banach spaces that are continuous with respect to the strong operator topology. A Kazhdan projection in $C_{\mathcal{F}}$ is a real self-adjoint idempotent $p \in C_{\mathcal{F}}$, such that for every $\pi \in \mathcal{F}, \pi : G \rightarrow GL(X)$, $\pi (p)$ is a projection on the space of invariant vectors under the action of $G$, i.e., $\pi (p)$ is a projection on  the space $X^\pi$.
%\end{definition}
Next, we shall define the notion of strong Banach property (T) introduced by Lafforgue in \cite{Lafforgue1}, \cite{Lafforgue2}. \\
For a locally compact group $G$, a length over $G$ is a continuous function $l: G \rightarrow \mathbb{R}_+$ such that $l(g) = l(g^{-1})$ and $l(g_1 g_2) \leq l(g_1) + l(g_2)$ for every $g,g_1,g_2 \in G$. Notice that if $l$ is a length $l$ over $G$, then for any $s \geq 0, c \geq 0$, $sl+c$ is also a length over $G$. For a family $\mathcal{E}$ of Banach spaces and a length function $l$, denote $\mathcal{F} (\mathcal{E},l)$ to be the family of representations $\pi$ on some $X \in \mathcal{E}$ such that $\Vert \pi (g) \Vert \leq e^{l(g)}$ for every $g \in G$. Recall that for any representation $\pi$ on $X$, $X^\pi$ denotes the subspaces of invariant vectors under the action of $\pi$, i.e., 
$$X^\pi = \lbrace v \in X : \forall g \in G, \pi (g).v =v \rbrace .$$

\begin{definition}
A group $G$ has strong Banach property (T)  if for any class of Banach spaces $\mathcal{E}$ of type $>1$ that is stable under duality and under complex conjugation and for any  length $l$ over $G$, there exists $s_0 >0$ such that for every $c \geq 0$, there $p \in C_{\mathcal{F} (\mathcal{E},s_0 l+c)}$ that is a real, self-adjoint idempotent such that for every $\pi \in \mathcal{F} (\mathcal{E},s_0 l+c)$, $\pi (p)$ is a projection on $X^\pi$.
\end{definition}
Our results does not achieve strong Banach property (T), but a slightly weaker notion that we shall call robust property (T). We define it as follows:
\begin{definition}
\label{robust property T definition}
A group $G$ has robust Banach property (T) with respect to a class of Banach spaces $\mathcal{E}$, if for any length $l$ over $G$, there exists $s_0 >0$ and a sequence of real functions $f_n \in C_c (G)$ with the following properties:
\begin{enumerate}
\item For every $n$, $f_n$ is symmetric, i.e., $f_n (g) = f_n (g^{-1})$.
\item For every $n$, $\int f_n=1$.
\end{enumerate}
And such that the sequence $(f_n)$ convergences in $C_{\mathcal{F} (\mathcal{E},s_0 l)}$ to $p$ and $\forall \pi \in \mathcal{F} (\mathcal{E},s_0 l)$, $\pi (p)$ is a projection on $X^\pi$.
\end{definition}
If we assume that $G$ is compactly generated (e.g., if $G$ has property (T)), we can give an equivalent definition that is more convenient:
\begin{definition}
\label{robust property T definition - compact generation}
Let $G$ be a compactly generated group and let $K$ be some symmetric compact set that generates $G$. For a class of Banach spaces $\mathcal{E}$ and a constant $s_0 \geq 0$, denote $\mathcal{F} (\mathcal{E}, K,s_0)$ to be the class of all the representations $\pi$ of $G$ on some $X \in \mathcal{E}$ such that $\sup_{g \in K} \Vert \pi (g) \Vert \leq e^{s_0}$. 

$G$ has robust Banach property (T) with respect to a class of Banach spaces $\mathcal{E}$, if there exists $s_0 >0$ and a sequence of real functions $f_n \in C_c (G)$ with the following properties:
\begin{enumerate}
\item For every $n$, $f_n$ is symmetric, i.e., $f_n (g) = f_n (g^{-1})$.
\item For every $n$, $\int f_n=1$.
\end{enumerate}
And such that the sequence $(f_n)$ convergences in $C_{\mathcal{F} (\mathcal{E},K,s_0 )}$ to $p$ and $\forall \pi \in \mathcal{F} (\mathcal{E},K,s_0)$, $\pi (p)$ is a projection on $X^\pi$.
\end{definition}
We'll leave the proof of equivalence between the two definitions to the reader. Note that in the above definition, $s_0$ depends on the choice of the generating set $K$, but the fact that $G$ has robust Banach property (T) with respect to $\mathcal{E}$ does not depend on this choice. 
\begin{remark}
The criteria we'll give below for robust Banach property (T) assumes compact generation and therefore definition \ref{robust property T definition - compact generation} is more convenient. We shall use the more general definition \ref{robust property T definition}, only when proving the general implications of robust Banach property (T) in the appendix.  
\end{remark}

\begin{remark}
The reader should note that $\mathcal{E}$ in the definition above is not assumed to be closed under duality or complex conjugation. Instead, we added the conditions that the functions $f_n$ are all symmetric and real. This definition was inspired by the definition given by de la Salle in \cite{Salle} for strong Banach property (T) with respect to a class of Banach spaces.
\end{remark}

\begin{remark}
If $G$ has robust Banach property (T) with respect to a class $\mathcal{E}$ that is closed under duality and complex conjugation, then the projection $p$ is a central idempotent and therefore a Kazhdan projection (see \cite{DNowak} and reference therein for details on Kazhdan projections). 
\end{remark}

\begin{remark}
Property (T) and the equivalent property (FH) are usually considered rigid, i.e., they are preserved under small changes. Following this line of thought, de la Salle has recently proven \cite{Salle2} that any group with property (T) will have robust Banach property (T) with respect to a class of Banach spaces that are all small enough deformations of Hilbert spaces. However, the reader should note that different groups with property (T) will allow different extents of deformation. Therefore, the question of what is $\mathcal{E}$ in the definition above for a given group remains interesting.
\end{remark}

Our main achievement is establishing criteria for robust Banach property (T) for a class of Banach spaces for groups $G$ with the following structure: $G$ is generated by compact subgroups $K_1,...,K_N$ such that each pair $K_i,K_j$ generate a compact group in $G$. This set up is quite general and apply to a several families of groups such as groups acting on Buildings (under certain assumptions on the action) and Kac-Moody-Steinberg groups defined in \cite{ErshovJZ}. Our criteria relays on bounding the "angle" between $K_i$ and $K_j$ in every unitary representation (on Hilbert spaces) of $\langle K_i, K_j \rangle$  (see exact formulation in theorems \ref{mSBT from angles in unitary rep theorem}, \ref{mSBT via Schatten norms}) below). This approach is heavily influenced by the work of Ershov and Jaikin-Zapirain in \cite{ErshovJZ} and the work of Kassabov in \cite{Kassabov} regrading property (T) via the notion of angle between groups. As far as we can tell, the  approach taken in this paper is different from the one taken by Lafforgue and his successors (although it seems to share some common features).  We also owe a great debt to the work of de la Salle \cite{Salle}, which provided the necessary machinery that allows us to use information regarding the unitary representations on Hilbert spaces to deduce information regarding representations on Banach spaces.

In order to prove our main results, we also establish a criterion for the convergence of the averaged projections method in Banach spaces that may be of independent interest. 

Robust Banach property (T) has two nice applications (which are the same applications derived in \cite{Lafforgue2} for strong property (T)). 
\subsection{Fixed point property}
We recall the following definitions given in \cite{BFGM}: 
\begin{definition}
Let $G$ be a topological group and let $X$ be a Banach space. $G$ is said to have property $F_X$ if every continuous affine isometric action of $G$ on $X$ has a fixed point. \\
For $p \in [1, \infty]$, $G$ is said to have property $F_{L^p}$ if every continuous affine isometric action of $G$ on any $L^p$ space has a fixed point.
\end{definition}
We note that it was proven in \cite{BFGM} and \cite{BGM} that if $G$ has property (T), then $G$ has property $F_{L^p}$ for every $p \in [1,2]$. 

Robust Banach property (T) can be used to prove property $F_X$ using the following proposition:

\begin{proposition}
\label{fixed point proposition}
Let $X$ be a Banach space and $G$ be a locally compact group. If $G$ has robust  Banach property (T) with respect to $\mathbb{C} \oplus X$ with the $l_2$ norm, then $G$ has property $F_X$.
\end{proposition}

The proof of this proposition needs an (easy) adaptation of the proof given in  \cite{Lafforgue2}. For completeness, the proof is provided in the appendix.

Using this application we are able to prove new fixed point theorems. For instance we prove a generalization of the following results:
\begin{theorem}
\label{groups acting on building theorem intro}
Let $\Sigma$ be a pure $n$-dimensional simplicial complex that is galley connected. Let $G$ be a group acting simplicially on a $\Sigma$ such that the action is cocompact and the fundamental domain $\Sigma / G$ is a single $n$-dimensional simplex and that the stabilizer of each $(n-2)$-dimensional simplex is a compact subgroup of $G$. Assume that every $1$-dimensional link of $\Sigma$ is a finite connected graph. Assume further that there is a constant $\eta > 1- \frac{1}{8(n+1)-11}$ such that for every $1$ dimensional link the smallest positive eigenvalue of the Laplacian on the link is $\geq \eta$. For every $p_0 >2$ there is a constant $1 > K (p_0)$ such that if $\eta \geq K(p_0)$, then $G$ has property $F_{L^p}$ for every $p \in [1,p_0)$.
\end{theorem}

\begin{theorem}
\label{Steinberg group theorem intro}
For any Banach space $X$ of type $p_1$ and cotype $p_2$ such that $\frac{1}{p_1} - \frac{1}{p_2} < \frac{1}{4}$ and any $n \geq 3, m \geq 1$, the Steinberg group $St_n (\mathbb{F}_q [t_1,...,t_m])$ has property $F_X$ provided that $q$ is sufficiently large.
\end{theorem}

\begin{theorem}
\label{Steinberg group theorem intro2}
For any $p_0 >2$, $n \geq 3$, $m \geq 1$ there is a constant $q(p,n,m)$ such that for any prime $q \geq q(p,n,m)$, Steinberg group $St_n (\mathbb{F}_q [t_1,...,t_m])$ has property $F_{L^p}$  for every $p \in [1,p_0)$.
\end{theorem}
More examples (and a more general statement of the above examples) are given in the last section of this paper. 

\subsection{Expander graphs}
\begin{definition}
Let $X$ be a Banach space and $\lbrace (V_i,E_i) \rbrace_{i \in \mathbb{N}}$ be a sequence of finite graphs with uniformly bounded degree, such that $\lim_i \vert V_i \vert = \infty$. We say that $\lbrace (V_i,E_i) \rbrace_{i \in \mathbb{N}}$ has a uniform coarse embedding in $X$ if there are functions $\phi_i : V_i \rightarrow X$ and functions $\rho_-, \rho_+: \mathbb{N} \rightarrow \mathbb{R}$ such that $\lim_n \rho_- (n) = \infty$ and 
$$\forall i \in \mathbb{N}, \forall x,y \in V_i, \rho_- (d_i (x,y)) \leq \Vert \phi_i (x) - \phi_i (y) \Vert \leq \rho_+ (d_i (x,y)) ,$$
where $d_i (x,y)$ is the graph distance in $(V_i,E_i)$ between $x$ and $y$.  \\
If $\lbrace (V_i,E_i) \rbrace_{i \in \mathbb{N}}$ has no uniform coarse embedding in $X$, we shall say that $\lbrace (V_i,E_i) \rbrace_{i \in \mathbb{N}}$ is a family of $X$-expanders.
\end{definition}

\begin{proposition}
\label{expander proposition}
Let $G$ be a finitely generated discrete group and let $\lbrace N_i \rbrace_{i \in \mathbb{N}}$ be a sequence of finite index normal subgroups of $G$ such that $\bigcap_i N_i = \lbrace 1 \rbrace$. Let $\mathcal{E}$ be a class of Banach spaces that is closed under $l_2$ sums. Fix $S$ to be some symmetric generating set of $G$. If $G$ has robust property (T) with respect to $\mathcal{E}$, then the family of Cayley graphs $\lbrace (G/N_i,S) \rbrace_{i \in \mathbb{N}}$ is a family of $X$-expanders for any $X \in \mathcal{E}$.   
\end{proposition}
The proof of this proposition is similar to the proof given in \cite{Lafforgue2}. For completeness, the proof is provided in the appendix. 

Using this application we are able construct new examples of families of graph that are expanders with respect to large classes of (non supereflexive) Banach spaces. For instance we can show the following:
\begin{theorem}
\label{Expander construction intro}
For any Banach space $X$ of type $p_1$ and cotype $p_2$ such that $\frac{1}{p_1} - \frac{1}{p_2} < \frac{1}{4}$, we can construct an $X$-expander by taking Cayley graphs of quotients of the group $EL_n (\mathbb{F}_q [t_1,...,t_m])$ provided that $q$ is sufficiently large.
\end{theorem}
A more general statement of this example is given in the last section of this paper.  \\ \\
\textbf{Structure of this paper.} Section 2 includes all the needed background material. Section 3 is devoted to proving a criterion for quick convergence of the averaged projections method, relaying on the concept of an angle between projections. In section 4, we formulate and prove several criteria for robust Banach property (T). In section 5, we give examples of groups with robust Banach property (T) and construct Banach expanders. \\ \\
\textbf{Acknowledgements.} The author would like to thank Mikael de la Salle, Mikhail Ershov and Simeon Reich for reading a previous draft of this paper and adding their insightful comments. Most of the work was done while the author was a visiting assistant professor at the Ohio State University and the author thanks the university for its hospitality.  

\section{Background}
\subsection{Projections in a Banach space}  
Let $X$ be a Banach space. Recall that a projection $P$ is a bounded operator $P \in \mathcal{B} (X)$ such $P^2 =P$.  Note that $\Vert P \Vert \geq 1$ if $P \neq 0$. For subspaces $M, N$ of $X$, we'll say that $P$ is a projection on $M$ along $N$ if $P$ is a projection such that $Im (P) = M$, $ker(P)=N$. 

\subsection{The Banach-Mazur distance}
The Banach-Mazur distance measures a "distance" between finite dimensional Banach spaces:
\begin{definition}
Let $Y_1, Y_2$ be two isomorphic Banach spaces. The (multiplicative) Banach-Mazur distance between $Y_1$ and $Y_2$ is defined as 
$$d_{BM} (Y_1, Y_2) = \inf \lbrace \Vert T \Vert \Vert T^{-1} \Vert : T : Y_1 \rightarrow Y_2 \text{ is a linear isomorphism} \rbrace.$$
\end{definition}
%\begin{remark}
%\label{BM distance using T^-1 less than 1}
%We note that for every invertible $T : Y_1 \rightarrow Y_2$, we can always replace $T$ with $T'=\Vert T^{-1} \Vert T$ and $T^{-1}$ with $(T')^{-1} = \frac{1}{\Vert T^{-1} \Vert} T^{-1}$ and 
%$$\Vert T \Vert \Vert T^{-1} \Vert = \Vert T' \Vert \Vert (T')^{-1} \Vert = \Vert T' \Vert.$$ 
%Therefore, without loss of generality, one can define $d_{BM}$ as
%\begin{dmath*}
%d_{BM} (Y_1, Y_2) = {\inf \lbrace \Vert T \Vert  : T : Y_1 \rightarrow Y_2 \text{ is a linear isomorphism such that } \Vert T^{-1} \Vert=1   \rbrace} = 
%{\inf \lbrace \Vert T \Vert  : T : Y_1 \rightarrow Y_2 \text{ is a linear isomorphism such that } \Vert T^{-1} \Vert \leq 1   \rbrace}.
%\end{dmath*}
%\end{remark}
This distance has a multiplicative triangle inequality:
\begin{proposition}
Let $Y_1,Y_2,Y_3$ be isomorphic Banach spaces. Then
$$d_{BM}(Y_1,Y_3) \leq d_{BM}(Y_1,Y_2) d_{BM}(Y_2,Y_3).$$
\end{proposition} 

We leave the proof of the above proposition as an exercise to the reader.

\subsection{Type and cotype}
Let $X$ be a Banach space. For $1< p_1 \leq 2$, $X$ is said to have (Gaussian) type $p_1$, if there is a constant $T_{p_1}$, such that for $g_1,...,g_n$ independent standard Gaussian random variables on a probability space $(\Omega, P)$, we have that for every $x_1,...,x_n \in X$ the following holds:
$$\left( \int_0^1 \left\Vert \sum_{i=1} g_i (\omega) x_i  \right\Vert^2 dP \right)^{\frac{1}{2}} \leq T_{p_1} \left( \sum_{i=1}^n \Vert x_i \Vert^{p_1} \right)^{\frac{1}{p_1}}.$$
The minimal constant $T_{p_1}$ such that this inequality is fulfilled is denoted $T_{p_1} (X)$. \\
For $2 \leq p_2 < \infty$, $X$ is said to have (Gaussian) cotype $p_2$, if there is a constant $C_{p_2}$, such that for $g_1,...,g_n$ independent standard Gaussian random variables on a probability space $(\Omega, P)$, we have that for every $x_1,...,x_n \in X$ the following holds:
$$C_{p_2} \left( \sum_{i=1}^n \Vert x_i \Vert^{p_2} \right)^{\frac{1}{p_2}} \leq \left( \int_0^1 \left\Vert \sum_{i=1} g_i (\omega) x_i  \right\Vert^2 dP \right)^{\frac{1}{2}}  .$$
The minimal constant $C_{p_2}$ such that this inequality is fulfilled is denoted $C_{p_2} (X)$. 
We shall say that a class of Banach spaces, $\mathcal{E}$, is of type $>1$, if there is $p_1 >1$, $K>0$ such that every $X \in \mathcal{E}$ is of type $p_1$ with $K^{(q_1)} (X) \leq K$. 
%We shall also make use of the following result proven in \cite{Tomczak-Jaegermann}:
%\begin{theorem}
%\label{from type-cotype to BM}
%If $X$ is a Banach space with type $q_1$ and cotype $q_2$, then for any $Y$ that is an $m$-dimensional subspace of $X$, we have that
%$$d_{BM} (Y, l_2^m) \leq 4 K^{(q_1)} (X) K_{(q_2)} (X) m^{\frac{1}{q_1} - \frac{1}{q_2}}.$$
%\end{theorem}

\begin{remark}
We remark that the Gaussian type and cotype defined above are equivalent to the usual (Rademacher) type and cotype (see \cite{HistoryBanach}[pages 311-312] and reference therein). 
\end{remark}

\begin{remark}
Recall that a Banach space is called superreflexive if it is isomorphic to a uniformly convex Banach space. 
In \cite{PisierXu}, Pisier and Xu showed that for any $p_2 >2$ one can construct a non superreflexive Banach space $X$ with type $2$ and cotype $p_2$.
\end{remark}

%This question was dealt with in the work of de la Salle in \cite{Salle} and de Laat and de la Salle in \cite{LaatSalle} and the estimate we'll give below (lemmas REF1, REF2) are due to these papers. We'll start by recalling the some terminology. 

\subsection{Vector valued $L^2$ spaces} 
\label{Vector valued spaces section}
Given a measure space $(\Omega, \mu)$ and Banach space $X$, a function $s : \Omega \rightarrow X$ is called simple if it is of the form:
$$s(\omega) = \sum_{i=1}^n \chi_{E_i} (\omega) v_i,$$
where $\lbrace E_1,...,E_n \rbrace$ is a partition of $\Omega$ where each $E_i$ is a measurable set, $\chi_{E_i}$ is the indicator function on $E_i$ and $v_i \in X$. \\
A function $f : \Omega \rightarrow X$ is called Bochner measurable if it is almost everywhere the limit of simple functions. Denote $L^2 (\Omega ; X)$ to be the space of Bochner measurable functions such that 
$$\forall f \in L^2 (\Omega ; X), \Vert f \Vert_{L^2 (\Omega ; X)} = \left( \int_\Omega \Vert f (\omega) \Vert^2_X d \mu (\omega) \right)^{\frac{1}{2}} < \infty.$$ 
Given an operator $T \in B(L^2 (\Omega, \mu))$, we can define $T \otimes id_X \in B(L^2 (\Omega ; X))$ by defining it first on simple functions. We shall need to following facts:
\begin{lemma}\cite{Salle}[Lemma 3.1]
\label{L2 norm stability}
Let $(\Omega, \mu)$ be a measure space, $C\geq 0$ and $T$ a bounded operator on $L^2 (\Omega, \mu)$. The class of Banach spaces $X$, for which $\Vert T \otimes id_X \Vert \leq C$ is stable under quotients, subspaces, $l_2$-sums and ultraproducts.
\end{lemma}

\begin{remark}
The fact that the above class is closed under $l_2$ sums, did not appear in \cite{Salle}[Lemma 3.1] and it is left as an exercise to the reader.
\end{remark}

\begin{lemma}
\label{norm of T otimes id using BM}
Let $(\Omega, \mu)$ be a measure space and $T$ a bounded operator on $L^2 (\Omega, \mu)$. Given two isomorphic Banach spaces $X$, $X'$, we have that 
$$\Vert T \otimes id_{X'} \Vert \leq d_{BM} (X,X') \Vert T \otimes id_{X} \Vert.$$ 
\end{lemma}

\begin{proof}
Let $S:X \rightarrow X'$ be an isomorphism, then 
$$T \otimes id_{X'} = (id_{L^2 (\Omega, \mu)} \otimes S ) \circ (T \otimes id_{X}) \circ (id_{L^2 (\Omega, \mu)} \otimes S^{-1} )$$ 
and therefore 
$$\Vert T \otimes id_{X'} \Vert \leq \Vert id_{L^2 (\Omega, \mu)} \otimes S \Vert \Vert T \otimes id_{X} \Vert \Vert  id_{L^2 (\Omega, \mu)} \otimes S^{-1} \Vert = (\Vert S \Vert \Vert S^{-1} \Vert) \Vert \Vert T \otimes id_{X} \Vert,$$
and the conclusion of the lemma follows.
\end{proof}

\subsection{Interpolation}

Two Banach spaces $X_0, X_1$ form a compatible pair $(X_0,X_1)$ if there is a continuous linear embedding of both $X_0$ and $X_1$ in the same topological vector space. The idea of complex interpolation is that given a compatible pair $(X_0,X_1)$ and a constant $0 < \theta < 1$, there is a method to produce a new Banach space $[X_0, X_1]_\theta$ as a "combination" of $X_0$ and $X_1$. We will not review this method here, and the interested reader can find more information on interpolation in \cite{Salle}[Section 2.4] and reference therein. The only fact we shall use regarding complex interpolation is the following:
\begin{lemma} \cite{Salle}[Lemma 3.1]
\label{interpolation fact}
Given a compatible pair $(X_0,X_1)$, a measure space $(\Omega,\mu)$ and an operator $T \in B(L^2 (\Omega,\mu))$, we have for every $0 \leq \theta \leq 1$ that
$$\Vert T \otimes id_{[X_0, X_1]_\theta} \Vert_{B(L^2 (\Omega ; [X_0, X_1]_\theta))} \leq \Vert T \otimes id_{X_0} \Vert_{B(L^2 (\Omega ; X_0))}^{1-\theta} \Vert T \otimes id_{X_1} \Vert_{B(L^2 (\Omega ; X_1))}^{\theta}.$$
\end{lemma}

\subsection{$\theta$-Hilbertian Banach spaces} 
The following definition is due to Pisier in \cite{Pisier}: a Banach space $X$ is called strictly $\theta$-Hilbertian for $0 < \theta \leq 1$, if there is a compatible pair $(X_0,X_1)$ such that $X_1$ is a Hilbert space such that $X = [X_0, X_1]_\theta$. Examples of are $L^p$ space and non-commutative $L^p$ spaces when  (in these cases $\theta = \frac{2}{p}$ if $2 \leq  p  < \infty $ and $\theta = 2 - \frac{2}{p}$ if $1 < p \leq 2$). In \cite{Pisier}, Pisier showed that any superreflexive Banach lattice is strictly $\theta$-Hilbertian and conjectured that any superreflexive Banach space is a subspace of a quotient of a strictly $\theta$-Hilbertian Banach space for some $\theta >0$.

\subsection{Group representations in a Banach space} 
Let $G$ be a locally compact group and $X$ a Banach space. Let $\pi$ be a representation $\pi :  G \rightarrow \mathcal{B} (X)$. Throughout this paper we shall always assume $\pi$ is continuous with respect to the strong operator topology without explicitly mentioning it. \\
Denote by $C_c (G)$ the groups algebra of compactly supported continuous functions on $G$ with convolution. For any $f \in C_c (G)$ we can define $\pi (f) \in \mathcal{B} (X)$ as
$$\forall v \in X, \pi (f). v = \int_G f(g) \pi(g).v  d\mu (g),$$
where the above integral is the Bochner integral with respect to the (left) Haar measure $\mu$ of $G$. \\
Recall that given $\pi$ one can define the following representations:
\begin{enumerate}
\item The complex conjugation of $\pi$, denoted $\overline{\pi} : G \rightarrow \mathcal{B} (\overline{X})$ is defined as 
$$\overline{\pi} (g). \overline{v} = \overline{\pi (g). v}, \forall g \in G, \overline{v} \in \overline{X}.$$  
\item The dual representation $\pi^* : G \rightarrow \mathcal{B} (X^*)$ is defined as 
$$\langle v, \pi^* (g) u  \rangle =  \langle \pi (g^{-1}) .v, u  \rangle, \forall g \in G, v \in X, u \in X^*.$$ 
\end{enumerate}

Next, we'll restrict ourselves to the case of compact groups. Let $K$ be a compact group with a Haar measure $\mu$ and let $C_c (K) = C(K)$ defined as above. Let $X$ be Banach space and let $\pi$ be a representation of $K$ on $X$ that is continuous with respect to the strong operator topology. We shall show that for every $f \in C_c (K)$ we can bound the norm of $\pi (f)$ using the norm of $\lambda \otimes id_X \in B(L^2 (K ;X))$ (the definition of $L^2 (K; X)$ is given in subsection \ref{Vector valued spaces section} above).
We shall start with the following result of de la Salle that deals with the case in which $\pi$ is an isometric representation:
\begin{proposition}\cite{Salle}[Proposition 2.7]
\label{bounding the norm of pi(f) - Salle}
Let $\pi$ be an isometric representation of a compact group $K$ on a Banach space $X$. Then for any real function $f \in C_c (K)$ we have that
$$\Vert \pi (f) \Vert_{B(X)} \leq \Vert (\lambda \otimes id_X) (f) \Vert_{B(L^2 (K ; X))},$$
where $\lambda$ is the left regular representation of $K$. 
\end{proposition}

\begin{remark}
The above proposition in \cite{Salle} in phrased in the language of signed Borel measures on $K$ and not $f \in C_c (K)$, but this is equivalent, since $f$ can be thought of as the density function of a signed Borel measure.
\end{remark}

We can use the above proposition to bound the norm of $\pi (f)$  in the more general case in which $\sup_{g \in K} \Vert \pi (g) \Vert < \infty$:

\begin{corollary}
\label{bounding the norm of pi(f) - corollary}
Let $\pi$ be a representation of a compact group $K$ on a Banach space $X$. Assume that $\sup_{g \in K} \Vert \pi (g) \Vert < \infty$, then for any real function $f \in C_c (G)$ we have that
$$\Vert \pi (f) \Vert_{B(X)} \leq \left( \sup_{g \in K} \Vert \pi (g) \Vert \right)^2  \Vert (\lambda \otimes id_X) (f) \Vert_{B(L^2 (K ; X))},$$
where $\lambda$ is the left regular representation of $G$. 
\end{corollary}

\begin{proof}
Define the following norm $\Vert . \Vert'$ on $X$:
$$\forall v \in X, \Vert v \Vert' = \sup_{g \in K} \Vert \pi (g).v \Vert_X.$$
Denote $X'$ to be the Banach space $X' = (X, \Vert . \Vert')$, then $\pi$ is an isometric representation on $X'$ and $d_{BM} (X,X') \leq \sup_{g \in K} \Vert \pi (g) \Vert$. By proposition \ref{bounding the norm of pi(f) - Salle}, we have that 
$$\Vert \pi (f) \Vert_{B(X')} \leq \Vert (\lambda \otimes id_{X'}) (f) \Vert_{B(L^2 (K ; X'))}.$$
By lemma \ref{norm of T otimes id using BM}, we have that
$$\Vert (\lambda \otimes id_{X'}) (f) \Vert_{B(L^2 (K ; X'))} \leq   \left( \sup_{g \in K} \Vert \pi (g) \Vert \right) \Vert (\lambda \otimes id_X) (f) \Vert_{B(L^2 (K ; X))} .$$
One can also check that by the definition of $\Vert . \Vert'$, we have that 
$$\Vert \pi (f) \Vert_{B(X)}  \leq   \left( \sup_{g \in K} \Vert \pi (g) \Vert \right) \Vert \pi (f) \Vert_{B(X')},$$
and the conclusion of the corollary follows.
\end{proof}

%Let $\mu$ a Haar measure on $H$ and denote by $\lambda$ the left regular representation of $H$ on $L^2 (H,\mu)$. Note that for every real function $f \in C_c (H)$ and for every Banach space $X$, we have that $\lambda (f) \in B (L^2 (H,\mu))$ and $\lambda_X (f) \in B(L^2 (H;X))$ is of the form $\lambda_X (f) = \lambda (f) \otimes id_X$.

%Denote also $X_\pi \subseteq X$ to be the subspace $X_\pi = ((X^*)^{\pi^*})^\perp$, i.e., 
%$$X_\pi = \lbrace v \in X : \forall v^* \in (X^{*})^{\pi^*}, \langle v,v^* \rangle = 0 \rbrace .$$
%\begin{definition}
%A representation $\pi :  G \rightarrow \mathcal{B} (X)$ is called complemented if $X = X^\pi + X_\pi $.
%\end{definition}

\subsection{Spectra of bipartite graphs}

Let $\mathcal{G} = (V,E)$ be a graph without loops, multiple edges or isolated vertices. Recall the following definitions:
\begin{definition}
The normalized Laplacian on $(V,E)$ as the operator $\Delta$ acting on $L^2 (V)$ as 
$$\Delta \psi (v) = \psi (v) - \frac{1}{d(v)} \sum_{u \in V, \lbrace u, v \rbrace \in E} \psi (u) ,$$
where $d(v)$ is the valency of $v$. $\Delta$ is positive operator with respect to the following inner product on $L^2 (V)$:
$$\langle \psi_1, \psi_2 \rangle = \sum_{v \in V} d(v) \psi_1 (v) \psi_2 (v).$$

If $\mathcal{G}$ is connected, then $0$ is an eigenvalue of multiplicity $1$ of $\Delta$ and the smallest positive eigenvalue of $\Delta$ is called the spectral gap of the Laplacian.
\end{definition}

\begin{definition}
The graph $\mathcal{G}$ is called bipartite if there are non empty disjoint sets $V_1 , V_2 \subset V$ such that $V = V_1 \cup V_2$ and $E \subseteq \lbrace \lbrace v_1 , v_2 \rbrace : v_1 \in V_1, v_2 \in V_2 \rbrace$, i.e., there are no edges between two vertices of $V_1$ or between two vertices of $V_2$. A bipartite graph is called semi-regular if for every $i=1,2$ and every $v, v' \in V_i$ we have $d(v) =d(v')$. 
\end{definition}

The following proposition is well known and the proof is a simple computation which is left for the reader:
\begin{proposition}
\label{bipartite graph spectrum proposition 1}
Let $\mathcal{G} = (V,E)$ be a bipartite graph with $V_1,V_2 \subset V$ as in the definition above. If $\psi \in L^2 (V)$ is an eigenvector of Laplacian $\Delta$ with eigenvalue $\eta$, then 
$$\psi' (v) = \begin{cases}
\psi(v) & v \in V_1 \\
- \psi(v) & v \in V_2
\end{cases},$$
is an eigenvector of $\Delta$ with eigenvalue $2-\eta$.
\end{proposition}

\begin{proposition}
\label{bipartite graph spectrum proposition 2}
Let $\mathcal{G} = (V,E)$ be a bipartite graph with $V_1,V_2 \subset V$ as in the definition above. If $\vert V_2 \vert < \vert V_1 \vert$, then the space of eigenfunctions of $\Delta$ with the eigenvalue $1$ has a subspace of functions supported on $V_1$ of dimension $\geq \vert V_1 \vert - \vert V_2 \vert$. 
\end{proposition}

\begin{proof}
Define 
$$W = \lbrace \psi \in L^2 (V) : \forall v \in V_2, \psi(v) = 0 \text{ and } \sum_{u \in V_1, \lbrace v ,u \rbrace \in E} \psi(u) =0\rbrace.$$
We conclude by noticing that for every $\psi \in W$, $\Delta \psi = \psi$ and that $W$ is a subspace of $L^2 (V)$ defined by $2 \vert V_2 \vert$ linear equations on $\vert V_1 \vert + \vert V_2 \vert$ variables and therefore the dimension of $W$ is at least $\vert V_1 \vert - \vert V_2 \vert$.     
\end{proof}

\begin{proposition}
\label{bipartite graph spectrum proposition 3}
Let $\mathcal{G} = (V,E)$ be a semi-regular bipartite graph and let $\psi \in L^2 (V)$ be an eigenfunction of $\Delta$ with an eigenvalue $0 < \eta <2$, then for $i=1,2$ we have 
$$\sum_{v \in V_i} \psi(v) = 0.$$
\end{proposition}

\begin{proof}
Let $\chi_{V}, \chi_{V_1}, \chi_{V_2}$ be the indicator functions of $V,V_1,V_2$. By definition of $\Delta$ we have that $\Delta \chi_V = 0$. Therefore by proposition \ref{bipartite graph spectrum proposition 1}, we get that $\Delta (\chi_{V_1} - \chi_{V_2}) = 2 (\chi_{V_1} - \chi_{V_2})$. From the fact that $\Delta$ is self-adjoint we deduce that $\langle \psi, \chi_V \rangle = \langle \psi, \chi_{V_1} - \chi_{V_2} \rangle =0$. Denote by $d_{V_1},d_{V_2}$ the valency of all the vertices in $V_1$ and $V_2$ respectively. Then 
$$0 = \langle \psi, \chi_V \rangle = d_{V_1} \sum_{v \in V_1} \psi(v) + d_{V_2} \sum_{v \in V_2} \psi(v),$$
$$0 = \langle \psi, \chi_{V_1} - \chi_{V_2} \rangle = d_{V_1} \sum_{v \in V_1} \psi(v) - d_{V_2} \sum_{v \in V_2} \psi(v),$$
and the proposition follows.
\end{proof}

\section{Averaged projections in a Banach space}

Given a family of projections $P_1,...,P_N$ on  $M_1,...,M_N$ in $X$, there is a well known algorithm of finding a projection on $\cap_{k=1}^N M_k$, which is known as the method of alternating projections. This algorithm can be stated (in full generality) as follows: let $\mathcal{S} (P_1,...,P_N)$  be the convex hull of the semigroup
consisting of all products with factors from $\lbrace P_1,...,P_N \rbrace$. Let $T \in \mathcal{S} (P_1,...,P_N)$, such that for every $k=1,...,n$, $P_k$ appears (in some product) in the decomposition of $T$ in $\mathcal{S} (P_1,...,P_N)$. Then under some assumptions on $X$ (for instance, if $X$ is uniformly convex and $\Vert P_k \Vert = 1$ for every $k$ as in \cite{BruckReich}), we have that the sequence $T^n$ converges in the strong operator topology to an operator $T^\infty$ that is a projection on $\cap_{k=1}^N M_k$. Below we shall restrict ourselves to a special case of the general alternating projections method described above in which $T= \frac{P_1+...+P_N}{N}$. This case is called the averaged projections method.\\
The rate of converges of this method is either "very slow" or "very fast". To make this statement precise we recall the following definitions and results from \cite{BadeaGrivauxM}:
\begin{definition}[Rates of convergence]
Let $X$ be a Banach space and let $(T_n)$ be a sequence of bounded operators on $X$. Assume that $(T_n)$ converges in the strong operator topology to $T^\infty \in \mathcal{B} (X)$. Then we say that:
\begin{enumerate}
\item Quick uniform convergence condition holds (abbreviated: (QUC) holds) if there are constants $C\geq 0$, $0 \leq r  <1$ such that $\Vert T^\infty - T_n \Vert \leq C r^n.$
\item Arbitrarily slow convergence condition holds (abbreviated: (ASC) holds) if for every sequence of positive numbers $(a_n)$ such that $\lim_{n \rightarrow \infty} (a_n)=0$ and for every $\varepsilon >0$, there exists $v \in X$ such that $\Vert v \Vert < \sup_n a_n + \varepsilon$ and for every $n$, 
$\Vert T^\infty v - T_n v \Vert \geq a_n.$
\end{enumerate}
\end{definition}

\begin{remark}
Note that in the above definition if (ASC) holds, then in particular $(T_n)$ does not converge to $T^\infty$ in the uniform operator topology. Indeed, for any arbitrary $n_0$, one can choose a positive sequence $(a_n)$ such that $\lim_{n \rightarrow \infty} (a_n)=0$, $\sup_n a_n <1$ and $a_{n_0} = \frac{1}{2}$. Therefore, by (ASC), there is $v \in X$, such that $\Vert v \Vert \leq 1$ and $\Vert T^\infty v - T_{n_0} v \Vert \geq \frac{1}{2}$. This is true for any $n_0$ and therefore we get that for any $n$, $\Vert T^\infty - T_n \Vert \geq \frac{1}{2}$.
\end{remark}

The conditions (QUC) and (ASC) represent the two extreme cases of convergence rates of sequences of bounded operators (assuming convergence in the strong operator topology). The theorem below states that in the case that $T_n$ is defined as $T_n = T^n$ for some fixed $T \in \mathcal{B} (X)$, these are the only possibilities of convergence:

\begin{theorem}[Dichotomy theorem]\cite{BadeaGrivauxM}[Theorem 2.1]
\label{Dichotomy thm}
Let $X$ be a Banach space and let $T \in \mathcal{B} (X)$. If $T^n$ converges in the strong operator topology to $T^\infty \in \mathcal{B} (X)$ then either (QUC) holds or (ASC) holds.  
\end{theorem}
Next, we want to establish a criterion that guarantees that the convergence is quick uniform in the case of averaged projections where $T=\frac{P_1+...+P_N}{N}$ (this case is considered because it will suit our needs later and because it is simple). In the case where $X=H$ is a Hilbert space, criteria for quick uniform are given in \cite{BadeaGrivauxM} and \cite{PRZ} (see also \cite{Kassabov} for related results) in terms of angles between subspaces. The basic idea is that if the angles between the subspaces are large enough, then the averaged projections method has quick uniform convergence. The concept of angle considered in the articles mentioned above is the Friedrichs angle defined as follows:
\begin{definition}[Friedrichs angle]
Let $M_1, M_2$ be subspaces of a Hilbert space $H$. Denote $M = M_1 \cap M_2$. If $M_1 =M$ or $M_2 = M$, Friedrichs angle between $M_1$ and $M_2$ is defined to be $\frac{\pi}{2}$. Otherwise, the Friedrichs angle between $M_1$ and $M_2$ is defined as:
$$\angle (M_1,M_2) = \arccos \left(
 \sup \left\lbrace \frac{\vert \langle u , v \rangle \vert}{\Vert u \Vert \Vert v \Vert} :0 \neq u \in M_1 \cap M^\perp,0 \neq v \in M_2 \cap M^\perp \right\rbrace  \right).$$
This definition is equivalent to 
$$\angle (M_1,M_2) = \arccos \left(\sup \lbrace \Vert P_{M_1} (v) \Vert : v \in M_2 \cap M^\perp, \Vert v \Vert =1 \rbrace \right),$$
where $P_{M_1}$ is the orthogonal projection on $M_1$. We'll leave this equivalence as an exercise to the reader.
\end{definition}
Although several authors gave definitions for angle between subspaces in Banach spaces (see for instance \cite{Ostrovski}), we could not find such a definition that suits our purpose and allows giving a criterion for quick uniform convergence of the averaged projections method in Banach spaces. There seems to be two major problems with such definitions: first, the lack of the concept orthogonality in Banach spaces (without passing to the dual space) and second, the fact that the projections are not uniquely determined by their image (unlike the case of orthogonal projections in a Hilbert space). In order to circumvent both problems, we shall define an angle between projections and not between subspaces. 
\begin{definition}[Friedrichs angle between projections]
Let $X$ be a Banach space and let $P_1, P_2$ be projections on $M_1,M_2$ respectively. Assume that there is a projection $P_{1,2}$ on $M_1 \cap M_2$ such that $P_{1,2} P_1 = P_{1,2}$ and $P_{1,2} P_2 = P_{1,2}$ and define 
$$\cos (\angle (P_1,P_2)) = \left(\max \left\lbrace \Vert P_1 (P_2 - P_{1,2} ) \Vert, \Vert P_2 (P_1 - P_{1,2} ) \Vert  \right\rbrace \right).$$ 
\end{definition}

\begin{remark}
In the above definition, we are actually defining the ``cosine'' of the angle. This is a little misleading, because in some cases  $\cos (\angle (P_1,P_2)) > 1$. 
\end{remark}

\begin{remark}
Note that the assumptions on $P_{1,2}$ above imply that $P_1 - P_{1,2}$ and $P_2- P_{1,2}$ are projections, i.e., $(P_1 - P_{1,2})^2 = P_1 - P_{1,2}$ and  $(P_2 - P_{1,2})^2 = P_2 - P_{1,2}$.
\end{remark}

\begin{remark}
In the case where $X$ is a Hilbert space and $P_1,P_2$ are orthogonal projections on $M_1,M_2$. The orthogonal projection $P_{1,2}$ on  $M_1 \cap M_2$ will always fulfil $P_{1,2} P_1 = P_{1,2}$ and $P_{1,2} P_2 = P_{1,2}$. Further more, note that in this case, we have that 
$$v \in M_2 \cap (M_1 \cap M_2)^\perp \Leftrightarrow (P_2-P_{1,2})v = v.$$
Therefore
\begin{dmath*}
\cos (\angle (M_1, M_2)) = {\sup \lbrace \Vert P_{1} (v) \Vert : v \in M_2 \cap M^\perp, \Vert v \Vert =1 \rbrace} = 
{\sup \lbrace \Vert P_{1} (v) \Vert : (P_2-P_{1,2})v = v, \Vert v \Vert =1 \rbrace} =  {\sup \lbrace \Vert P_{1} (P_2-P_{1,2})v \Vert : (P_2-P_{1,2})v = v, \Vert v \Vert =1 \rbrace} = \Vert P_{1} (P_2-P_{1,2}) \Vert = \cos (\angle (P_1, P_2)) . 
\end{dmath*}
Therefore the above definition is a complete analogue to the definition of Friedrichs angle in Hilbert spaces.
\end{remark}

Next, we shall also need the following useful constant $a(P_1,P_2)$:

\begin{definition}
Let $X$ be a Banach space and let  $P_1, P_2$ be projections on $M_1,M_2$ respectively. Define $a(P_1,P_2)$ as follows:
\begin{align*}
a (P_1, P_2 ) = \inf \lbrace \gamma \in [0, \infty ): \forall v \in X, \Vert (P_1 P_2 - P_2 P_1 ) v \Vert \leq \gamma \Vert (P_1 - P_2)  v \Vert  \rbrace.
\end{align*}
\end{definition}

Let $P_1,...,P_N$ be projections in a Banach space $X$. For $T=\frac{P_1+...+P_N}{N}$, we will show that if for every couple $P_i, P_j$, $\cos (\angle (P_i,P_j))$ is small enough, then the sequence $\lbrace T^n \rbrace$ converges uniformly quickly to $\bigcap_{i=1}^N Im (P_i)$. We shall prove this in two steps: first, we will show that if for every $P_i, P_j$, $a(P_i,P_j)$ is small enough then the sequence $\lbrace T^n \rbrace$ converges uniformly quickly to $\bigcap_{i=1}^N Im (P_i)$. Second, we will show $a(P_i,P_j)$ can be bounded from above by (a function of) $\cos (\angle (P_i,P_j))$.

We start the first step by introducing the following function $E: X \rightarrow \mathbb{R}$:
$$E (v) = \sum_{1 \leq i < j \leq N} \Vert (P_i -P_j) v \Vert.$$
\begin{lemma}
\label{E lemma}
Let $X$ be a Banach space and let $P_1,...,P_N$ be projections in $X$ (where $N \geq 2$).
%on $M_1,...,M_N$. Assume that for every $1 \leq i < j \leq N$, there is a projection $P_{ij}$ on $M_i \cap M_j$ that commutes with $P_i$ and with $P_j$. 
Denote $T = \frac{P_1+...+P_N}{N},$
$$\alpha = \max \lbrace  a(P_i,P_j) : 1 \leq i < j \leq N \rbrace,$$
$$\beta = \max \lbrace \Vert P_1 \Vert,..., \Vert P_N \Vert \rbrace.$$
Then for every $v \in X$ and every $n \in \mathbb{N}$, we have that
$$E(T^n v) \leq \left(\dfrac{1+(N-2)\beta + (2N-3) \alpha}{N} \right)^n E(v),$$
and
$$E \left(T^n v \right) \leq 2 \beta {N \choose 2} \left(\dfrac{1+(N-2)\beta + (2N-3) \alpha}{N} \right)^n \Vert v \Vert.$$
\end{lemma}  

\begin{proof}
We start by observing that for every $1 \leq i < j \leq N$ we have that
$$(P_i - P_j) (P_i + P_j) = P_i - P_j + P_i P_j - P_j P_j.$$
Therefore, for every $v \in X$, we have that
\begin{dmath}
\label{ineq for P_i+ P_j}
\Vert (P_i - P_j) (P_i + P_j) v \Vert \leq \Vert (P_i - P_j)v \Vert + \Vert (P_i P_j - P_j P_j)v \Vert \leq \\ \Vert (P_i - P_j)v \Vert + a(P_i, P_j) \Vert (P_i - P_j)v \Vert \leq (1+ \alpha)  \Vert (P_i - P_j)v \Vert
\end{dmath}
Let $i,j$ as before and let $1 \leq k \leq N$ such that $k \neq i,j$. Observe that
\begin{dmath*}
(P_i - P_j) P_k = P_i P_k - P_j P_k =P_i P_k - P_k P_i +P_k P_i - P_k P_j + P_k P_j - P_j P_k = 
(P_i P_k - P_k P_i) + (P_k P_j - P_j P_k) + P_k (P_i - P_j). 
\end{dmath*} 
Therefore, for every $v \in X$, we have that 
\begin{dmath}
\label{ineq for P_i, P_j, P_k}
{\Vert (P_i - P_j) P_k v \Vert \leq \Vert (P_i P_k - P_k P_i)  v \Vert + \Vert (P_k P_j - P_j P_k) v \Vert + \Vert  P_k (P_i - P_j) v \Vert \leq} \\ a(P_i,P_k) \Vert (P_i - P_k)v \Vert + a (P_j, P_k)  \Vert (P_j - P_k)v \Vert + \Vert P_k \Vert \Vert  (P_i - P_j) v \Vert \leq \\ \alpha (\Vert (P_i - P_k)v \Vert  + \Vert (P_j - P_k)v \Vert) + \beta \Vert  (P_i - P_j) v \Vert.
\end{dmath}
Using \eqref{ineq for P_i+ P_j}, \eqref{ineq for P_i, P_j, P_k}, we get that 
\begin{dmath*}
\Vert (P_i - P_j) T v \Vert = \left\Vert (P_i - P_j) (\frac{P_1+...+P_N}{N}) v \right\Vert \leq \\ \dfrac{1+(N-2)\beta + \alpha}{N} \Vert (P_i - P_j)v \Vert + \dfrac{\alpha}{N} \sum_{1 \leq k \neq i, j \leq N} \left( \Vert (P_i - P_k)v \Vert  + \Vert (P_j - P_k)v \Vert \right).
\end{dmath*}
By summing this inequalities along all $1 \leq i < j \leq N$, we get that
$$E \left( T v \right)\leq \left(\dfrac{1+(N-2)\beta + (2N-3) \alpha}{N} \right) E (v).$$
Therefore, by induction we have for every $n \in \mathbb{N}$ that 
$$E \left(T^n v \right) \leq \left(\dfrac{1+(N-2)\beta + (2N-3) \alpha}{N} \right)^n E (v).$$
To conclude, notice that for every $v$ we have that 
$$\Vert E (v) \Vert \leq {N \choose 2} 2 \beta \Vert v \Vert.$$
\end{proof}

\begin{lemma}
\label{Quick uniform convergence a lemma}
Let $X$ be a Banach space and let $P_1,...,P_N$ be projections in $X$ (where $N \geq 2$).
Denote $T=\frac{P_1+...+P_N}{N}$.
If there are $ \alpha < \frac{1}{2N-3} \text{ and } \beta < \frac{N-1-(2N-3)\alpha}{N-2},$ such that 
$$ \max \lbrace  a(P_i,P_j) : 1 \leq i < j \leq N \rbrace \leq \alpha,$$
$$\max \lbrace \Vert P_1 \Vert,..., \Vert P_N \Vert \rbrace \leq \beta,$$
then there are constants $r' = r' (\alpha, \beta), C' = C'(\alpha, \beta)$ and an operator $T^\infty$, such that $0 \leq r' <1, C' \geq 0$ and $\Vert T^\infty - T^n  \Vert \leq C' (r')^n$. Moreover, $T^\infty$ is a projection on $\bigcap_{i=1}^N Im (P_i)$.
\end{lemma}

\begin{proof}
Note that for every $v \in X$, we have that 
\begin{dmath*}
\Vert (T - T^2)v \Vert = \left\Vert \sum_{i=1}^N \frac{P_i}{N} (P_i - T)v \right\Vert =  \\ \left\Vert \sum_{i=1}^N \dfrac{P_i}{N} \sum_{j=1}^N \dfrac{P_i - P_j}{N} v \right\Vert \leq \dfrac{\beta}{N^2} \sum_{1 \leq i , j \leq N} \Vert (P_i - P_j)v \Vert  = \\ \dfrac{2 \beta}{N^2} \sum_{1 \leq i < j \leq N} \Vert (P_i - P_j)v \Vert = \dfrac{2 \beta}{N^2} E(v).
\end{dmath*}
Therefore, for any $n \in \mathbb{N}$ and any $v \in X$ we have that 
\begin{dmath}
\label{T^(n+1)-T^n}
\Vert (T^{n+1} - T^n)v \Vert = \Vert (T^2-T) (T^{n-1} v) \Vert \leq \dfrac{2 \beta}{N^2} E(T^{n-1} v) \leq \\ \dfrac{(2 \beta)^2}{N^2} {N \choose 2} \left(\dfrac{1+(N-2)\beta + (2N-3) \alpha}{N} \right)^{n-1} \Vert v \Vert,
\end{dmath}
where the last inequality is due to lemma \ref{E lemma}. Denote 
$$r' = \dfrac{1+(N-2)\beta + (2N-3) \alpha}{N} .$$
Notice that the conditions  $ \alpha < \frac{1}{2N-3} \text{ and } \beta < \frac{N-1-(2N-3)\alpha}{N-2}$ stated in the theorem, insure that $r' <1$. After simplifying \eqref{T^(n+1)-T^n}, we get that
$$\Vert T^{n+1} - T^n \Vert \leq \dfrac{(N-1)2 \beta^2}{N}  (r')^{n-1}.$$
Therefore, for every two integers $m >n$, we have that
\begin{dmath*}
\Vert T^m - T^n \Vert \leq \Vert T^m -T^{m-1} \Vert + \Vert T^{m-1} -T^{m-2} \Vert +...+  \Vert T^{n+1} -T^{n} \Vert \leq \\
 \dfrac{(N-1)2 \beta^2}{N} ((r')^{m-2} +(r')^{m-3} +...+(r')^{n-1} ) \leq \dfrac{(N-1)2 \beta^2}{N} \dfrac{1}{(1-r')r'} (r')^n.
\end{dmath*}
Denote $C' = \frac{(N-1)2 \beta^2}{N} \frac{1}{(1-r')r'}$. We showed that $(T^n)$ is a Cauchy sequence with respect to the operator norm and therefore converges to an operator $T^\infty$, and
$$\Vert T^\infty - T^n \Vert \leq C' (r')^n.$$
One can easily verify that $(T^\infty)^2 = T^\infty$ and therefore $T^\infty$ is a projection. To see that $Im (T^\infty) =  \bigcap_{i=1}^N Im (P_i)$, we first note that for every $v \in \bigcap_{i=1}^N Im (P_i)$, we have that $Tv =v$ and therefore $\bigcap_{i=1}^N Im (P_i) \subseteq Im (T^\infty)$. To finish the proof, we need to show that  $Im (T^\infty) \subseteq \bigcap_{i=1}^N Im (P_i)$. Note that  $T T^\infty = T^\infty$ and therefore for every $v \in Im (T^\infty)$, we have that $T v =v$.  From lemma \ref{E lemma}, we have that for every $v \in Im (T^\infty)$, $ E ( v ) = E (T v) \leq  r' E(v)$ and therefore  $E(v) =0$ (recall that $0 \leq r' <1$).
$E(v)$ is defined as 
$$E (v) = \sum_{1 \leq i < j \leq N} \Vert (P_i -P_j) v \Vert.$$
Therefore $E(v)=0$ implies that $ P_1 v = P_2 v = ... = P_N v$.
To finish, we'll again use the fact that $Tv =v$ and get that for every $1 \leq i \leq N$,
$$\left( v = Tv = \dfrac{P_1 v +... + P_N v}{N} = \dfrac{P_i v + ...+ P_i v}{N} = P_i v \right) \Rightarrow v \in Im (P_i).$$
Since this is true for all the $i$'s, we got that $Im (T^\infty) \subseteq \bigcap_{i=1}^N Im (P_i)$.
\end{proof}

Next, we turn to the second step, bounding $a(P_i,P_j)$ as a function of $\cos (\angle (P_i,P_j))$:
\begin{lemma}
\label{bounding a(P1,P2) by the angle lemma}
Let $X$ be a Banach space and let $P_1,P_2$ be two projections in $X$, such that there is a projection $P_{1,2}$ on $Im (P_1) \cap Im (P_2)$ such that $P_{1,2} P_1 = P_{1,2}$ and $P_{1,2} P_2 = P_{1,2}$. Denote $\beta = \max \lbrace \Vert P_1 \Vert, \Vert P_2 \Vert \rbrace$. Then 
$$a(P_1,P_2) \leq \dfrac{2 (1+ \beta) \cos (\angle (P_1,P_2))}{1-\cos (\angle (P_1,P_2))}.$$ 
\end{lemma}

\begin{proof}
Let $v \in X$. We start by noting that 
\begin{dmath*}
(P_1 P_2 - P_2 P_1)v = (P_1 P_2 - P_{1,2}) v - (P_2 P_1 - P_{1,2}) v  = P_1 (P_2 - P_{1,2})(P_2 - P_{1,2})v - P_2 (P_{1} - P_{1,2}) (P_{1} - P_{1,2})v .
\end{dmath*}
This yields that
\begin{dmath*}
{\Vert (P_1 P_2 - P_2 P_1)v \Vert \leq \cos (\angle (P_1,P_2)) \left(\Vert (P_1 - P_{1,2})v \Vert +\Vert (P_2 - P_{1,2})v \Vert\right)} \leq \\ 2\cos (\angle (P_1,P_2)) \max \lbrace \Vert (P_1 - P_{1,2})v \Vert,\Vert (P_2 - P_{1,2})v \Vert \rbrace.
\end{dmath*}
Therefore in order to prove the inequality stated in the lemma it is enough to show that 
$$\max \lbrace \Vert (P_1 - P_{1,2})v \Vert,\Vert (P_2 - P_{1,2})v \Vert \rbrace \leq \dfrac{1+\beta}{1-\cos (\angle (P_1,P_2))}\Vert (P_1 - P_2)v \Vert.$$
Assume without loss of generality that $\Vert (P_1 - P_{1,2})v \Vert \geq \Vert (P_2 - P_{1,2})v \Vert$. Then we have that
\begin{dmath*}
\Vert (P_1-P_2)v \Vert = \Vert (I-P_2)(P_1-P_2)v + P_2 (P_1-P_2)v  \Vert \geq \\ \Vert (I-P_2)(P_1-P_2)v \Vert - \Vert P_2 (P_1-P_2)v  \Vert \geq \\ \Vert (I-P_2)(P_1-P_{1,2} + P_{1,2} - P_2)v \Vert -  \beta \Vert (P_1-P_2)v  \Vert = \\  \Vert (I-P_2)(P_1-P_{1,2})v \Vert -  \beta \Vert (P_1-P_2)v  \Vert \geq \\ \Vert (P_1-P_{1,2})v \Vert -\Vert P_2(P_1-P_{1,2})v \Vert - \beta \Vert (P_1-P_2)v  \Vert = \\ \Vert (P_1-P_{1,2})v \Vert -\Vert P_2(P_1-P_{1,2})(P_1-P_{1,2})v \Vert - \beta \Vert (P_1-P_2)v  \Vert \geq \\ (1-\cos (\angle (P_1,P_2))\Vert (P_1-P_{1,2})v \Vert -\beta \Vert (P_1-P_2)v  \Vert,        
\end{dmath*}
which yields the necessary inequality to finish the proof.
\end{proof}

Combining the two lemmas above gives raise to the following convergence criterion:
\begin{theorem}
\label{Quick uniform convergence criterion}
Let $X$ be a Banach space and let $P_1,...,P_N$ be projections in $X$ (where $N \geq 2$).
Denote $T=\frac{P_1+...+P_N}{N}$,
$$\cos_\max = \max \lbrace  \cos(\angle (P_i,P_j)) : 1 \leq i < j \leq N \rbrace.$$
Assume there are constants 
$$\gamma < \frac{1}{8N-11} \text{ and } \beta < 1+ \frac{1-(8N-11)\gamma}{N-2 + (3N-4)\gamma},$$
such that  $ \max \lbrace \Vert P_1 \Vert,..., \Vert P_N \Vert \rbrace \leq \beta$ and $\cos_\max \leq \gamma$.
Then there are constants $r = r(\gamma, \beta), C = C(\gamma, \beta)$ and an operator $T^\infty$, such that $0 \leq r <1, C \geq 0$ and $\Vert T^\infty - T^n  \Vert \leq C r^n$. Moreover, $T^\infty$ is a projection on $\bigcap_{i=1}^N Im (P_i)$.
\end{theorem}

\begin{proof}
By lemma \ref{bounding a(P1,P2) by the angle lemma}, we get that
$$\max \lbrace  a(P_i,P_j) : 1 \leq i < j \leq N \rbrace \leq  \dfrac{2 (1+ \beta) \cos_\max}{1-\cos_\max} \leq \dfrac{2 (1+ \beta) \gamma}{1-\gamma} .$$
Denote $\alpha = \frac{2 (1+ \beta) \gamma}{1-\gamma}$. By lemma \ref{Quick uniform convergence a lemma}, it is enough to verify that 
$$\alpha < \dfrac{1}{2N-3} \text{ and } \beta < \frac{N-1-(2N-3)\alpha}{N-2},$$
and then take $C = C' ( \frac{2 (1+ \beta) \gamma}{1-\gamma}, \beta)$, $r = r' ( \frac{2 (1+ \beta) \gamma}{1-\gamma}, \beta)$, where $C'$ and $r'$ are the constants given in lemma \ref{Quick uniform convergence a lemma}. 

By our notations $ \alpha < \frac{1}{2N-3}$, is equivalent to
$$ \dfrac{2 (1+ \beta) \gamma}{1-\gamma} < \dfrac{1}{2N-3}.$$
Standard algebraic manipulations yields that the above inequality is equivalent to 
$$\beta < 1+ \dfrac{1-(8N-11)\gamma}{(4N-6)\gamma}.$$
First note that $\beta \geq 1$ and therefore without the assumption  $\gamma <\frac{1}{8N-11}$ this inequality cannot hold. Second note that $\gamma <\frac{1}{8N-11} <1$ and therefore the assumption that $\beta < 1+ \frac{1-(8N-11)\gamma}{N-2 + (3N-4) \gamma}$ implies the needed inequality. 

Next, we need to check that 
$\beta < \frac{N-1-(2N-3)\alpha}{N-2}$, i.e., we need to check that
$$\beta < \frac{N-1-(2N-3)\frac{2 (1+ \beta) \gamma}{1-\gamma}}{N-2}.$$
Standard algebraic manipulations yields that this is equivalent to 
$$ \beta < 1+ \frac{1-(8N-11)\gamma}{N-2 + (3N-4)\gamma},$$
as needed.
\end{proof}

\section{Robust Banach property (T)}
Throughout this section we will work under the following assumptions (and notations): $G$ is a locally compact group with a Haar measure $\mu$. Assume that $G$ generated by compact subgroups $K_1,...,K_N$, such that $\mu (K_i) >0$ for every $i=1,...,N$ and such that for each $i \neq j$, $K_{i,j}=\langle K_i, K_j \rangle$ is compact. Denote $K=\bigcup_{1 \leq i < j \leq N} K_{i,j}$. Define functions $k_i \in C_c (G), \forall 1 \leq i \leq N$ and $k_{i,j} \in C_c (G), \forall 1 \leq i < j \leq N$ as 
$$k_i (g) = \begin{cases}
\frac{1}{\mu (K_i)} & g \in K_i \\
0 & g \notin K_i
\end{cases},$$
$$k_{i,j} (g) = \begin{cases}
\frac{1}{\mu (K_{i,j})} & g \in K_{i,j} \\
0 & g \notin K_{i,j}
\end{cases}.$$

\begin{proposition}
\label{pi (k_i) are projections}
Let $G$ be as above and let $X$ be a Banach space. For any $1 \leq i \leq N$ and any representation $\pi$ of $G$ on $X$, $\pi (k_i)$ is a projection on $X^{\pi (K_i)} = \lbrace v \in X : \forall g \in K_i, \pi (g).v=v \rbrace$. Further more, for any $1 \leq i <j \leq N$ and any representation $\pi$ of $G$ on $X$, $\pi (k_{i,j})$ is a projection on $X^{\pi (K_{i,j})} = X^{\pi (K_i)} \cap X^{\pi (K_j)}$ and $\pi (k_{i,j}) \pi (k_i) =\pi (k_{i,j}), \pi (k_{i,j}) \pi (k_j) =\pi (k_{i,j})$. 
\end{proposition}

\begin{proof}
Fix some $i$. Note that for every $g \in K_i$, $g. k_i = k_i$. This implies two things: first,  $(k_i)^2 = k_i$ and therefore $\pi (k_i)^2 = \pi (k_i)$, i.e., $\pi (k_i)$ is a projection. Second, for every $g \in K_i$ and every $v \in X$, $\pi (g) \pi (k_i). v =  \pi (k_i). v$ and therefore $Im (\pi (k_i)) \subseteq X^{\pi (K_i)}$.  To see that $X^{\pi (K_i)} \subseteq Im (\pi (k_i))$, notice that for every $v \in X^{\pi (K_i)}$ we have that
$$\pi (k_i).v = \int_{g \in K_i}  \dfrac{\pi (g).v}{\mu (K_i)} d \mu (g) = \dfrac{1}{\mu (K_i)} \int_{g \in K_i} v d \mu (g) = v.$$
The proof that $\pi (k_{i,j})$ is a projection on $X^{\pi (K_{i,j})}$ is similar and therefore is left for the reader. To see that  $\pi (k_{i,j}) \pi (k_i) =\pi (k_{i,j}), \pi (k_{i,j}) \pi (k_j) =\pi (k_{i,j})$, we note that in the algebra $C_c (G)$, we have that for any $g \in K_{i,j}$, $k_{i,j} k_i = k_{i,j}, k_{i,j} k_j = k_{i,j}$ and therefore this equality passes to any representation $\pi$.
%To conclude, notice that $\pi$ is an isometric representation and therefore for any $v \in X$ we have that
%\begin{dmath*}
%\Vert \pi (k_i).v \Vert = \Vert \sum_{g \in K_i}  \dfrac{\pi (g).v}{\vert K_i \vert} \Vert \leq  \sum_{g \in K_i} \dfrac{\Vert \pi (g).v  \Vert}{\vert K_i \vert}  =  \sum_{g \in K_i} \dfrac{\Vert v  \Vert}{\vert K_i \vert}  = \Vert v \Vert,
%\end{dmath*}
%i.e., $\pi (k_i)$ is of norm $1$.
\end{proof}
The above proposition shows that for any representation $\pi$ of $G$, we can define $\cos (\angle (\pi (k_i),\pi (k_j)))$ as in the previous section, i.e, 
$$\cos (\angle (\pi (k_i),\pi (k_j))) = \max \lbrace \Vert \pi (k_i) \pi (k_j) - \pi (k_{i,j} ) \Vert,  \Vert \pi (k_j) \pi (k_i) - \pi (k_{i,j} ) \Vert  \rbrace.$$
Note that this yields that
$$\cos (\angle (\pi (k_i),\pi (k_j))) = \max \lbrace \Vert \pi (k_i k_j - k_{i,j} ) \Vert,  \Vert \pi (k_j k_i - k_{i,j} ) \Vert \rbrace.$$
In particular, this is true when $G=K_{i,j}$ and therefore $\cos (\angle (\pi (k_i),\pi (k_j)))$ is defined for any representation of $K_{i,j}$. For any $1 \leq i < j \leq N$ and any Banach space $X$, denote by $\lambda_{i,j}$ the left regular representation on $L^2 (K_{i,j}, \mu)$ and recall that $\lambda_{i,j} \otimes id_X$ is an isometric representation on $L^2 (K_{i,j} ; X)$. Denote
$$\cos^X (\angle (k_i, k_j)) = \cos (\angle ((\lambda_{i,j} \otimes id_X) (k_i), (\lambda_{i,j} \otimes id_X) (k_j))),$$
$$\cos^X_\max = \max_{1 \leq i < j \leq N} \cos^X (\angle (k_i, k_j)).$$
\begin{remark}
As noted above 
$$\cos^X (\angle (k_i, k_j)) = \max \lbrace \Vert (\lambda_{i,j} \otimes id_X) (k_i k_j - k_{i,j} ) \Vert,  \Vert (\lambda_{i,j} \otimes id_X) (k_j k_i - k_{i,j} ) \Vert \rbrace.$$
Therefore $\cos^X (\angle (k_i, k_j))$ is given as a maximum of two operator of the general form $T \otimes id_X$, where $T$ is an operator on $T \in B(L^2 (K_{i,j} ,\mu)$ and therefore we can apply the results for vector value $L^2$ spaces to bound $\cos^X (\angle (k_i, k_j))$.
\end{remark}
For a class $\mathcal{E}$ of Banach spaces denote 
$$\cos^{\mathcal{E}}_\max= \sup_{X \in \mathcal{E}} \cos^X_\max.$$
With this notation, we can use the criterion stated in theorem \ref{Quick uniform convergence criterion} to get a criterion for robust Banach property (T):
\begin{theorem}
\label{mSBT criterion via regular rep}
Let $G$,$K$ be as above and let $\mathcal{E}$ be a class of Banach spaces. Assume that there is $\varepsilon >0$ such that 
$$\cos^{\mathcal{E}}_\max  \leq \frac{1- \varepsilon}{8N-11} .$$ 
Then there is  $s_0 >0$ such that the sequence $(\frac{k_1 + ... +k_N}{N} )^n$ converges to $p$  in $C_{\mathcal{F} (\mathcal{E},K,s_0)}$  as $n \rightarrow \infty$ and $\forall \pi \in \mathcal{F} (\mathcal{E},K, s_0)$, $\pi (p)$ is a projection on $X^\pi$.
\end{theorem}

\begin{proof}
Assume without loss of generality that $\varepsilon <1$. \\
\textbf{Step 1:} We'll show that there is $s_1 >0$ such that for every $\pi \in \mathcal{F} (\mathcal{E},K, s_1)$, we have that
\begin{equation}
\label{ineq of step 1 in criterion thm}
\cos_\max^\pi = \max_{1 \leq i < j \leq N} \cos (\angle (\pi (k_i),\pi (k_j))) \leq \dfrac{1- \frac{\varepsilon}{2}}{8N-11} 
\end{equation}
By definition for every $s_1 >0$ and every $\pi \in \mathcal{F} (\mathcal{E},K, s_1)$ we have that 
$$\forall 1 \leq i < j \leq N, \sup_{g \in K_{i,j}} \Vert \pi (g) \Vert \leq e^{s_1} .$$
Combining the above inequality with corollary \ref{bounding the norm of pi(f) - corollary}, yields that for every $f \in C_c (K_{i,j})$ we have that 
$$\Vert \pi (f) \Vert_{B(X)} \leq e^{2 s_1} \Vert (\lambda_{i,j} \otimes id_X) (f) \Vert_{B(L^2 (K_{i,j} ;X))}.$$
By the definition of $\cos (\angle (\pi (k_i),\pi (k_j)))$ and $\cos^X (\angle (k_i, k_j))$, this yields that for every $1 \leq i < j \leq N$ we have that
$$\cos (\angle (\pi (k_i),\pi (k_j))) \leq e^{2 s_1} \cos^X (\angle (k_i, k_j)) \leq e^{2 s_1} \left( \frac{1- \varepsilon}{8N-11} \right).$$  
Therefore choosing 
$$s_1 = \dfrac{\ln (1 + \frac{\varepsilon}{2(1-\varepsilon)})}{2},$$
yields inequality \eqref{ineq of step 1 in criterion thm} as needed. \\
\textbf{Step 2:} We'll show that there is $s_2 >0$ such that for every $\pi \in \mathcal{F} (\mathcal{E},K, s_2)$, we have that
\begin{equation}
\label{ineq of step 2 in criterion thm}
 \max_{1 \leq i \leq N} \Vert \pi (k_i) \Vert < 1 + \dfrac{\varepsilon}{8N}.
\end{equation}
By definition for every $s_2 >0$ and every $\pi \in \mathcal{F} (\mathcal{E},K, s_2)$, we have that
$$\forall 1 \leq i \leq N, \sup_{g \in K_i} \Vert \pi (g) \Vert \leq e^{s_2}.$$
By the definition of the functions $k_i$, this yields that 
$$\forall 1 \leq i \leq N, \Vert \pi (k_i) \Vert \leq e^{s_2}.$$
Therefore choosing 
$$s_2 = \ln (1 + \frac{1}{8N}),$$
yields inequality \eqref{ineq of step 2 in criterion thm} as needed. \\
\textbf{Step 3:} To finish, choose $s_0 = \min \lbrace s_1, s_2 \rbrace >0$. Denote
$$\gamma = \dfrac{1- \frac{\varepsilon}{2}}{8N-11}, \beta = 1 + \dfrac{\varepsilon}{8N}.$$
For every $\pi \in \mathcal{F} (\mathcal{E},K,s_0)$, we have that by previous steps
$$\cos_\max^\pi \leq \gamma,$$
$$\max_{1 \leq i \leq N} \Vert \pi (k_i) \Vert \leq \beta.$$
Also note that 
$$\beta \leq 1+ \dfrac{\varepsilon}{8N} < 1+ \dfrac{\varepsilon}{8N-12} \leq 1 + \dfrac{\frac{\varepsilon}{2}}{4N-6} \leq 1+ \dfrac{1 - (8N-11)\gamma}{N-2 + (3N-4)\gamma}.$$
This implies that the conditions of theorem \ref{Quick uniform convergence criterion} are fulfilled for the projections $\pi (k_1),...,\pi (k_N)$ and therefore  there are $0 \leq r = r(\gamma,\beta) <1 , C = C (\gamma,\beta) \geq 0$ and an operator which we denote as $\pi (p)$ such that 
$$\left\Vert \pi (p) - \left( \dfrac{\pi (k_1) + ...+ \pi (k_N)}{N} \right)^n \right\Vert \leq C r^n,$$
and such that $\pi (p)$ is a projection on $\bigcap_{1 \leq i \leq N} Im (\pi (k_i)) = \bigcap_{1 \leq i \leq N} X^{\pi (K_i)} = X^\pi$ (the last equality is due to the fact that $K_1,...,K_N$ generate $G$). Note that the constants that bound the rate of convergence $r,C$ are independent of $\pi$ and therefore $\left( \frac{\pi (k_1) + ...+ \pi (k_N)}{N} \right)^n$ converges in $C_{\mathcal{F} (\mathcal{E},K, s_0)}$.
\end{proof}

%In order to apply the above theorem we need to find classes of Banach spaces $\mathcal{E}$ in which we can bound $\cos_\max^{\mathcal{E}}$, but we shall delay the answer to this question. 
Next we shall address the following question: let $\mathcal{E}$ be a class of Banach spaces such that $\cos_\max^{\mathcal{E}} \leq c$, how can we expand $\mathcal{E}$ to a larger class $\mathcal{E}'$ such that $\cos_\max^{\mathcal{E}'} \leq c'$ when $c'$ is a function of $c$. 

First, we note that after finding a class $\mathcal{E}$ with a bound on $\cos_\max^{\mathcal{E}}$, we can assume it is stable under certain operations. To be specific, given a class of Banach spaces $\mathcal{E}$, denote by $\overline{\mathcal{E}}$ to be the smallest class of Banach spaces that contains $\mathcal{E}$ and is stable under quotients, subspaces, $l_2$-sums, ultraproducts and complex interpolation for any $0 < \theta< 1$ of any compatible pair $(X_0,X_1)$ such that $X_0,X_1 \in \overline{\mathcal{E}}$. The next proposition states that a bound on $\cos_\max^{\mathcal{E}}$ implies a bound on $\cos_\max^{\overline{\mathcal{E}}}$:
\begin{proposition}
\label{closure of mathcal(E) prop}
Let $\mathcal{E}$ be a class of Banach spaces and let $c \geq 0$ be some constant. If $\cos_\max^{\mathcal{E}} \leq c$, then $\cos_\max^{\overline{\mathcal{E}}} \leq c$.
\end{proposition}

\begin{proof}
Combine the definition of $\cos_\max^X$ with lemma \ref{L2 norm stability} and lemma \ref{interpolation fact}.
\end{proof}

Second, we observe that considering a neighbourhood of $\mathcal{E}$ with respect to the Banach-Mazur distance changes $\cos_\max^{\mathcal{E}}$ by the radius of this neighbourhood. To be precise:

\begin{proposition}
\label{neigh. of mathcal(E) prop}
Let $\mathcal{E}$ be a class of Banach spaces and let $c \geq 0, \delta \geq 0$ be some constants. Let 
$B_{BM} (\mathcal{E},\delta)$ be the class of Banach spaces defined as:
$$B_{BM} (\mathcal{E},\delta) = \lbrace X : \exists Y \in \mathcal{E}, d_{BM} (X,Y) \leq 1 + \delta \rbrace.$$
If $\cos_\max^{\mathcal{E}} \leq c$, then $\cos_\max^{B_{BM} (\mathcal{E},\delta)} \leq c(1+\delta)$.
\end{proposition}

\begin{proof}
Combine the definition of $\cos_\max^X$ with lemma \ref{norm of T otimes id using BM}. 
\end{proof}

Third, we observe that taking $\theta$-interpolation of some $X \in \mathcal{E}$ changes $\cos_\max^{\mathcal{E}}$ as a function of $\theta$:

\begin{proposition}
\label{interpolation of mathcal(E) prop}
Let $\mathcal{E}$ be a class of Banach spaces and let $2 > c \geq 0, 0 <\theta \leq 1$ be some constants. Let $Int(\mathcal{E}, \geq \theta)$ be the class of Banach spaces defined as
$$Int(\mathcal{E}, \geq \theta) = \lbrace X : \exists X_1 \in \mathcal{E} \text{ and } X_0 \text{ such that } X = [X_0,X_1]_{\theta'} \text{ for some } \theta' \geq \theta \rbrace. $$
If $\cos_\max^{\mathcal{E}} \leq c$, then $\cos_\max^{Int(\mathcal{E}, \geq \theta)} \leq 2 \left( \frac{c}{2} \right)^{\theta}$.
\end{proposition}  

\begin{proof}
Note that for any Banach space $X$, we have for every $1 \leq i < j \leq N$ that
\begin{dmath*}
{\Vert (\lambda_{i,j} \otimes id_X) (k_i).v \Vert = \left\Vert \int_{g \in K_i}  \dfrac{(\lambda_{i,j} \otimes id_X) (g).v}{\mu (K_i)} d \mu (g) \right\Vert \leq}\\
 \int_{g \in K_i}  \dfrac{\Vert (\lambda_{i,j} \otimes id_X) (g).v \Vert}{\mu (K_i)} d \mu (g) = \int_{g \in K_i}  \dfrac{\Vert v \Vert}{\mu (K_i)} d \mu (g) = \Vert v \Vert. 
\end{dmath*} 
Therefore $$\Vert (\lambda_{i,j} \otimes id_X) (k_i) \Vert \leq 1$$ 
and similarly 
$$\Vert (\lambda_{i,j} \otimes id_X) (k_j) \Vert \leq 1, \Vert (\lambda_{i,j} \otimes id_X) (k_{i,j}) \Vert \leq 1.$$
This yields that for every $1 \leq i < j \leq N$ we have that 
$$\Vert \lambda_{i,j} (k_i k_j - k_{i,j}) \otimes id_X  \Vert \leq 2, \Vert \lambda_{i,j} (k_j k_i - k_{i,j}) \otimes id_X  \Vert \leq 2.$$ 
Let $X_1 \in \mathcal{E}$ and $X_0$ be a Banach space such that $(X_0,X_1)$ are a compatible pair. By lemma \ref{interpolation fact} we have that for every $\theta \leq \theta'  \leq 1$ and every $1 \leq i < j \leq N$ that 
\begin{dmath*}
{\Vert \lambda_{i,j} (k_i k_j - k_{i,j}) \otimes id_{[X_0,X_1]_{\theta '}}  \Vert \leq} \\ {\Vert \lambda_{i,j} (k_i k_j - k_{i,j}) \otimes id_{X_0}  \Vert^{1-\theta'} \Vert \lambda_{i,j} (k_i k_j - k_{i,j}) \otimes id_{X_0}  \Vert^{\theta'} \leq} \\ 2^{1-\theta'} c^{\theta '} \leq 2^{1-\theta} c^{\theta } = 2 \left(\dfrac{c}{2} \right)^{\theta}.   
\end{dmath*}
Similarly,
$$\Vert \lambda_{i,j} (k_j k_i - k_{i,j}) \otimes id_{[X_0,X_1]_{\theta '}}  \Vert \leq 2 \left(\dfrac{c}{2} \right)^{\theta}, $$
and we are done by the definition of $\cos_\max^{[X_0,X_1]_{\theta '}}$.
\end{proof}

Combining all the above propositions yields the following:
\begin{corollary}
\label{enlarging mathcal(E) corollary}
Let $G$ be as above and let $\mathcal{E}$ be a class of Banach spaces. Assume that there is a constant $c \geq 0$ such that
$$\cos^{\mathcal{E}}_\max  \leq c < \frac{1}{8N-11} .$$
Let $c'$ be a constant such that $c \leq c' <  \frac{1}{8N-11}$. Denote $\delta = \frac{c'}{c} -1, \theta = \frac{\ln (2)- \ln (c')}{\ln (2) - \ln (c) }$ and
$$\mathcal{E} ' = \overline{B_{BM} (\mathcal{E}, \delta) \cup Int(\mathcal{E}, \geq \theta)}.$$
Then
$$\cos^{\mathcal{E}'}_\max  \leq c' < \frac{1}{8N-11} ,$$
and there is  $s_0 >0$ such that the sequence $(\frac{k_1 + ... +k_N}{N} )^n$ converges to $p$  in $C_{\mathcal{F} (\mathcal{E}',K,s_0)}$  as $n \rightarrow \infty$ and $\forall \pi \in \mathcal{F} (\mathcal{E}',K, s_0)$, $\pi (p)$ is a projection on $X^\pi$.
\end{corollary}

\begin{proof}
Combing propositions \ref{closure of mathcal(E) prop}, \ref{neigh. of mathcal(E) prop}, \ref{interpolation of mathcal(E) prop} and theorem \ref{mSBT criterion via regular rep}.
\end{proof}

The above corollary gives us a way to get a class of Banach spaces $\mathcal{E}'$ for which $G$ has robust Banach property (T) providing that we have a class of of Banach spaces $\mathcal{E}$ such that $\cos^{\mathcal{E}}_\max $ is small. We are left with the question of how to produce such a class $\mathcal{E}$. Below we shall describe two methods to do so based on our knowledge of unitary representations of $K_{i,j}'s$ in Hilbert spaces. These methods can be summarized as follows:
\begin{itemize}
\item Method 1: take $\mathcal{E} = \mathcal{H}$ as the class of all Hilbert spaces. In this case $\cos^{\mathcal{H}}_\max $ can be bounded via analysing the classical Friedrichs angles between fixed subspaces in Hilbert spaces. This in turn can be done via analysing angles in irreducible representations of $K_{i,j}$.
\item Method 2: having knowledge on all the eigenvalues of $\pi (k_i k_j - k_{i,j})$ (and not just the norm) in any unitary representation $\pi$ allows us to pass to a richer class of Banach spaces via the Schatten norm of $\lambda_{i,j} (k_i k_j - k_{i,j})$. This method is taken from the work of de la Salle \cite{Salle}.
\end{itemize}
Next, we shall give a detailed account on each method.
\subsection{Robust Banach property (T) via bounding $\cos^{\mathcal{H}}_\max $}
Let $\mathcal{H}$ be the class of all Hilbert spaces. Bounding  $\cos^{\mathcal{H}}_\max $ is achieved by classical  Friedrichs angles in Hilbert spaces, by the following observation:
\begin{observation}
Notice that for any $H \in \mathcal{H}$, we have that for any $1 \leq i < j \leq N$, $L^2 (K_{i,j} ; H)$ is a Hilbert space and $\lambda_{i,j} \otimes id_H$ is a unitary representation on this space. Also note that since $\lambda_{i,j} (k_i) \otimes id_H$, $\lambda_{i,j} (k_i) \otimes id_H$, $\lambda_{i,j} (k_{i,j}) \otimes id_H$ are all projections of norm $1$ and therefore they are orthogonal projections. Therefore for any $1 \leq i < j \leq N$, bounding 
$$\cos^{\mathcal{H}} (\angle (k_i, k_j)) = \sup_{H \in \mathcal{H}} \cos^H (\angle (k_i, k_j))$$ 
boils down to bounding the (classical) Friedrichs angle $\cos (\angle (H^{\pi (K_i)}, H^{\pi (K_j)}))$ for any unitary representation $\pi$ on some Hilbert space $H$. 
\end{observation}

Combining the above observation with corollary \ref{enlarging mathcal(E) corollary} gives the following theorem:
\begin{theorem}
\label{mSBT from angles in unitary rep theorem}
Let $G$ be as above. Assume that there is some constant $c < \frac{1}{8N-11}$ such that for any $1 \leq i < j \leq N$ and any unitary representation of $K_{i,j}$ on a Hilbert space $H$ we have that $\cos (\angle (H^{\pi (K_i)}, H^{\pi (K_j)})) \leq c$. 
Let $c'$ be a constant such that $c \leq c' <  \frac{1}{8N-11}$. Denote $\delta = \frac{c'}{c} -1, \theta = \frac{\ln (2)- \ln (c')}{\ln (2) - \ln (c) }$ and
$$\mathcal{E}  = \overline{B_{BM} (\mathcal{H}, \delta) \cup Int(\mathcal{H}, \geq \theta)},$$
where $Int(\mathcal{H}, \geq \theta)$ is the class of all the $\theta'$-Hilbertian Banach spaces with $\theta ' \geq \theta$ and $B_{BM} (\mathcal{H}, \delta)$ is the class of all the spaces $X$ isomorphic to some Hilbert space $H=H (X)$ such that $d_{BM} (X,H) \leq 1+\delta$.
Then there is  $s_0 >0$ such that the sequence $(\frac{k_1 + ... +k_N}{N} )^n$ converges to $p$  in $C_{\mathcal{F} (\mathcal{E},K,s_0 )}$  as $n \rightarrow \infty$ and $\forall \pi \in \mathcal{F} (\mathcal{E},K, s_0)$, $\pi (p)$ is a projection on $X^\pi$. 
\end{theorem}

This theorem has a nice corollary regarding fixed point properties:
\begin{corollary}
\label{fixed point corollary from angles in unitary rep}
Let $G$ be as above. Assume that there is some constant $c < \frac{1}{8N-11}$ such that for any $1 \leq i < j \leq N$ and any unitary representation of $K_{i,j}$ on a Hilbert space $H$ we have that $\cos (\angle (H^{\pi (K_i)}, H^{\pi (K_j)})) \leq c$. Denote $\delta = \frac{1}{c(8N-11)} -1$, $\theta = \frac{\ln (2) + \ln (8N-11)}{\ln (2) - \ln (c)}$. If $X$ is a Banach space of one of the following types:
\begin{enumerate}
\item $X$ is isomorphic to a Hilbert space $H$ with $d_{BM} (X,H) < 1+\delta$.
\item $X$ is $\theta'$-Hilbertian with $\theta' > \theta$.
\end{enumerate}
Then $G$ has property $F_X$, i.e., every continuous affine isometric action of $G$ on $X$ has a fixed point.
\end{corollary}

\begin{proof}
Note that in the above theorem $\mathcal{E}$ contain every Hilbert space and in particular $\mathbb{C} \in \mathcal{E}$. Also note that $\mathcal{E}$ is closed under $l_2$ sums. Therefore we get the corollary by combining the above theorem with proposition \ref{fixed point proposition}.
\end{proof}

Last, we'll make two remark regarding bounding $\cos (\angle (H^{\pi (K_i)}, H^{\pi (K_j)}))$ for some fixed $1 \leq i < j \leq N$.
\begin{remark}
Observe that due to Peter-Weyl theorem, if for any irreducible unitary representation $\pi$ we have that 
$$\cos (\angle (H^{\pi (K_i)}, H^{\pi (K_j)})) \leq c,$$
then for any unitary representation $\pi$ we have
$$\cos (\angle (H^{\pi (K_i)}, H^{\pi (K_j)})) \leq c.$$
Therefore, it is enough to bound the angle for irreducible representations. 
\end{remark}

\subsection{Robust Banach property (T) via Schatten norms}
We start by recalling the following definitions: for a Hilbert space $H$ and a bounded operator $T \in B(H)$ and a constant $r \in [1,\infty]$, the $r$-th Schatten norm is defined as 
$$r < \infty, \Vert T \Vert_{S^r} = \left( \sum_{i=1}^\infty (s_i (T))^r \right)^{\frac{1}{r}},$$
$$\Vert T \Vert_{S^\infty} = \sup \lbrace s_1 (T), ...\rbrace,$$ 
where $s_1 (T) \geq s_2 (T) \geq ...$ are the eigenvalues of $\sqrt{T^* T}$. An operator $T$ is said to be of Schatten class $r$ if $\Vert T \Vert_{S^r} < \infty$. 

In \cite{Salle} the following proposition is proved:
\begin{proposition}\cite{Salle}[Proposition 3.3]
\label{Schatten norms prop from Salle}
Let $1 < p_1 \leq 2 \leq p_2 < \infty$ and let $r \in [2, \infty)$ such that $\frac{1}{p_1} - \frac{1}{p_2} < \frac{1}{r}$. There is a constant $M=M(p_1,p_2,r) \geq 0$ such that the following holds. If $X$ is a Banach space of type $p_1$ and cotype $p_2$,  $(\Omega, \mu)$ is a measure space and  $T \in B(L^2 (\Omega,\mu))$ of Schatten class $r$, then
$$\Vert T \otimes id_X \Vert_{B(L^2(\Omega ; X))} \leq M T_{p_1} (X) C_{p_2} (X) \Vert T \Vert_{S^r}.$$  
\end{proposition} 

\begin{remark}
The constant $M$ in the above proposition can be computed explicitly. To be precise
$$M= \sum_{i=1}^\infty 2^{\frac{r}{r-1} (\frac{1}{p_1} - \frac{1}{p_2} - \frac{1}{r} )i}  .$$
 \end{remark}
 
Using the above proposition gives us a way to bound $\cos^{\mathcal{E}}_\max$ using the Schatten norm of $\lambda (k_i k_j - k_{i,j}), \lambda (k_j k_i - k_{i,j})$ in $B(L^2 (K_{i,j}, \mu))$ for certain classes of Banach spaces. We shall need the following notation: for $1 < p_1 \leq 2 \leq p_2 < \infty$, $T_{p_1} \geq 1, C_{p_2} \geq 1$ constants denote $\mathcal{T} (p_1, p_2, T_{p_1}, C_{p_2} )$ to be the class of Banach spaces of type $p_1$ and cotype $p_2$ such that for every $X \in \mathcal{T} (p_1, p_2, T_{p_1}, C_{p_2} )$ we have that 
$T_{p_1} (X) \leq T_{p_1}, C_{p_2} (X) \leq C_{p_2}$.
\begin{proposition}
\label{Schatten norms prop for cos}
Let $G$ as above. Let $1 < p_1 \leq 2 \leq p_2 < \infty$, $T_{p_1} \geq 1, C_{p_2} \geq 1$ constants and let $r \in [2, \infty)$ such that $\frac{1}{p_1} - \frac{1}{p_2} < \frac{1}{r}$. Assume that for every $1 \leq i < j \leq N$ we have that $\lambda (k_i k_j - k_{i,j}) \in B(L^2 (K_{i,j}, \mu))$ are of Schatten class $r$. Denote 
$$\cos_\max^{S^r} = \max \lbrace \Vert \lambda (k_i k_j - k_{i,j}) \Vert_{S^r} : 1 \leq i < j \leq N  \rbrace.$$
Then for $M=M(p_1,p_2,r)$ as in the proposition above we have that 
$$\cos_\max^{\mathcal{T} (p_1, p_2, T_{p_1}, C_{p_2} )} \leq M T_{p_1} C_{p_2} \cos_\max^{S^r}.$$
\end{proposition}

\begin{proof}
Note that for every $1 \leq i < j \leq N$, $\lambda (k_i k_j - k_{i,j}) $ is the adjoint operator of $\lambda (k_j k_i - k_{i,j)}$, therefore $ \Vert \lambda (k_i k_j - k_{i,j}) \Vert_{S^r} =  \Vert \lambda (k_j k_i - k_{i,j}) \Vert_{S^r}$. Combine proposition \ref{Schatten norms prop from Salle} with the definition of $\cos^X_\max$.
\end{proof}

Combining the above proposition with theorem \ref{mSBT criterion via regular rep} gives the following result: 
\begin{theorem}
\label{mSBT via Schatten norms}
Let $G$ as above and let $1 < p_1 \leq 2 \leq p_2 < \infty$ and let $r \in [2, \infty)$ such that $\frac{1}{p_1} - \frac{1}{p_2} < \frac{1}{r}$. Denote 
$$M= \sum_{i=1}^\infty 2^{\frac{r}{r-1} (\frac{1}{p_1} - \frac{1}{p_2} - \frac{1}{r} )i}  .$$
Assume that there is a constant $c < \frac{1}{8N-11}$ such that $M \cos_\max^{S^r} \leq c$, and let $c'$ be a constant such that $c \leq c' < \frac{1}{8N-11}$. Denote
$$\mathcal{E} = \overline{\bigcup_{T_{p_1}, C_{p_2} \in [1, \infty), cT_{p_1} C_{p_2} \leq c'} Int \left(\mathcal{T} (p_1, p_2, T_{p_1}, C_{p_2} ), \geq \frac{\ln (2) - \ln (c')}{\ln (2) - \ln (cT_{p_1} C_{p_2})} \right)}.$$
Then there is  $s_0 >0$ such that the sequence $(\frac{k_1 + ... +k_N}{N} )^n$ converges to $p$  in $C_{\mathcal{F} (\mathcal{E},K,s_0)}$  as $n \rightarrow \infty$ and $\forall \pi \in \mathcal{F} (\mathcal{E},K, s_0)$, $\pi (p)$ is a projection on $X^\pi$. 
\end{theorem}

\begin{proof}
Combine the above proposition \ref{Schatten norms prop for cos}, theorem \ref{mSBT criterion via regular rep} and propositions \ref{closure of mathcal(E) prop}, \ref{interpolation of mathcal(E) prop}.
\end{proof}

Combining the above theorem with proposition \ref{fixed point proposition} yields the following corollary:
\begin{corollary}
\label{fixed point property via Schatten norms}
Assume the conditions of the above theorem hold and let $\mathcal{E}$ be as in the above theorem, then for any $X \in \mathcal{E}$, $G$ has property $F_X$, i.e., for every continuous affine isometric action of $G$ on $X$ has a fixed point.
\end{corollary}

\begin{remark}
One should note that due to Peter-Weyl theorem finding the eigenvalues of $\lambda ((k_i k_j -k_{i,j})^* (k_i k_j -k_{i,j}))$ in $L^2 (K_{i,j},\mu)$ in order to calculate that Schatten norm boils down to finding those eigenvalues in every irreducible representation of $K_{i,j}$.
\end{remark}

\subsection{Angle between groups and Schatten norm in Hilbert spaces using combinatorial data}
In the two methods described above we deduced robust Banach property (T) using knowledge of the spectrum of $\lambda ((k_i k_j - k_{i,j})^* (k_i k_j - k_{i,j}))$ for each $1 \leq i < j \leq N$. One way to obtain such knowledge is analysing the unitary representations of $K_{i,j}$. Below we present a more combinatorial way for analysing the spectrum of $\lambda ((k_i k_j - k_{i,j})^* (k_i k_j - k_{i,j}))$ by analysing the spectrum of the Laplacian of a graph constructed using $K_{i,j}, K_i, K_j$.

\begin{definition}
\label{graph of K_1, K_2 definition}
Let $K_{1,2}$ be a compact group and let $K_1, K_2$ be finite index subgroups of $K_{1,2}$ such that $K_{1,2} = \langle K_1, K_2 \rangle$. Define a bipartite graph $\mathcal{G}=(V,E)$ as follows:
\begin{itemize}
\item For $i=1,2$, $V_i$ is the set for right cosets:
$$V_i = \lbrace  K_i g : g \in K_{1,2} \rbrace,$$
and $V = V_1 \cup V_2$.
\item $K_1 g$ and $K_2 g'$ are connected by an edge if $K_1 g \cap K_2 g' \neq \emptyset$. In other words, if $l_i = [K_i : K_{1} \cap K_2]$ and $h_1^i,...,h_{l_i}^i$ are representatives in $K_i$ such that $K_i = \bigcup_{j} (K_1 \cap K_2) h^i_{j}$. Then $K_1 g$ is connected to $K_2 h^2_1 g,...,K_2 h^2_{l_2} g$ and $K_2 g$ is connected to $K_1 h^1_1 g,...,K_1 h^1_{l_1} g$. 
\end{itemize} 
\end{definition}
Notice that $\mathcal{G}$ above is semi-regular, since for every $i=1,2$ and every $v \in V_i$, $d(v) = l_i$. 

The next lemma connects the eigenvalues of the Laplacian on $\mathcal{G}$ to the eigenvalues of $\lambda ((k_1 k_2 - k_{1,2})^* (k_1 k_2 - k_{1,2}))$. A weaker form of this connection already appeared in \cite{DymaraJ}, where the spectral gap of the Laplacian on $\mathcal{G}$ was used to bound the norm of $\lambda ((k_1 k_2 - k_{1,2}))$.  

\begin{lemma}
\label{Laplacian eigenvalues lemma}
Let $K_{1,2}$, $K_1, K_2$ be as above and let $\lambda$ be the left regular representation on $L^2 (K_{1,2},\mu)$, where $\mu$ is the Haar measure of $K_{1,2}$. Let $\Delta$ be the graph Laplacian of $\mathcal{G}$ defined above. Let $0 = \eta_1 < \eta_2 \leq \eta_3 \leq ...\leq` \eta_{\vert V \vert -1} < \eta_{\vert V \vert} =2$ be the eigenvalues (including multiplicities) of $\Delta$. 
Then for $k_1,k_2, k_{1,2} \in C_c (K_{1,2})$ defined as in the previous section we have that the non-trivial eigenvalues of $\lambda ((k_1 k_2 - k_{1,2})^* (k_1 k_2 - k_{1,2}))$ are the non zero values of
$$(1- \eta_2)^2,...,(1-\eta_{\min \lbrace \vert V_1 \vert , \vert V_2 \vert \rbrace})^2,$$
accounting for multiplicities.
\end{lemma}

\begin{proof}
Assume without loss of generality that $[K_{1,2}:K_2] \leq [K_{1,2}:K_1]$, i.e., assume that $\vert V_2 \vert \leq \vert V_1 \vert$ (if $[K_{1,2}:K_2] > [K_{1,2}:K_1]$ we repeat the argument below for $k_2 k_1 - k_{1,2}$ instead of $k_1 k_2 - k_{1,2}$).

We start by exploring operators on $L^2(V)$ that we will later connect to $\lambda (k_1)$ and $\lambda (k_2)$. Abusing notation, we define for $i=1,2$ the operator $\chi_{V_i} \in L^2 (V)$ as multiplying by the indicator function $\chi_{V_i}$, i.e., $\chi_{V_i} \psi (v) = \chi_{V_i} (v)\psi(v)$. Define $M_1, M_2$ acting on $L^2 (V)$ as $M_1 = \chi_{V_{1}} (I- \Delta) \chi_{V_2}$ and $M_2 = \chi_{V_{2}} (I- \Delta) \chi_{V_1}$. In other words:
$$M_1 \psi (v) = \begin{cases}
0 & v \in V_2 \\
\frac{1}{d(v)} \sum_{u \in V_2, \lbrace v,u \rbrace \in E} \psi(u) & v \in V_1
\end{cases},$$
$$M_2 \psi (v) = \begin{cases}
0 & v \in V_1 \\
\frac{1}{d(v)} \sum_{u \in V_1, \lbrace v,u \rbrace \in E} \psi(u) & v \in V_2
\end{cases}.$$
By the definition of $M_1$ and $M_2$, we get that $M_2 M_1 = \chi_{V_2} (I-\Delta) \chi_{V_1} (I-\Delta)  \chi_{V_2}$. For every function $\psi \in L^2 (V)$, if $\psi$ is supported on $V_1$ then $M_2 M_1 f =0$. Therefore $M_2 M_1$ has at most $\vert V_2 \vert$ non zero eigenvalues (accounting for multiplicities). 
For an eigenfunction $\psi$ of $\Delta$ with $\Delta \psi = \eta \psi$ we have that $M_2 M_1 (\chi_{V_2} \psi) = (1-\eta)^2 (\chi_{V_2} \psi)$. Let $\psi_1 = \chi_V, \psi_2,...,\psi_{\vert V_2 \vert}$ be the eigenfunctions of $0=\eta_1,...,\eta_{\vert V_2 \vert}$. By propositions \ref{bipartite graph spectrum proposition 1}, \ref{bipartite graph spectrum proposition 2}, the space of functions supported on $V_2$ is spanned by $\chi_{V_2} \psi_1 = \chi_{V_2},\chi_{V_2} \psi_2,...,\chi_{V_2} \psi_{\vert V_2 \vert}$. 

Next, we connect $M_1$ and $M_2$ to the operators $\lambda (k_1)$ and $\lambda (k_2)$ acting on $L^2 (K_{1,2}, \mu)$. For $i=1,2$ we denote
$$L^2 (V)_i = \lbrace f \in L^2 (V) : supp (\psi) \subseteq V_i \rbrace,$$
$$L^2 (K_{1,2},\mu)_i = \lbrace f \in L^2 (K_{1,2}, \mu) : \forall g \in K_{1,2}, \forall g' , g'' \in K_i g, f(g')=f(g'') \rbrace.$$
By the definition of the vertex sets $V_1$ and $V_2$, there are natural identifications $F_i:  L^2 (K_{1,2},\mu)_i \rightarrow  L^2 (V)_i$. We will show that
$$\forall f \in L^2 (K_{1,2},\mu)_1, \lambda (k_2) f = F_2^{-1} M_2 F_1 f.$$

First, we note that for every $f \in L^2 (K_{1,2},\mu)$, we have that $\lambda (k_2) \in L^2 (K_{1,2},\mu)_2$. Fix $g,g'$ such that $K_2 g = K_2 g'$. There is $h' \in K_2$ such that $g=h' g'$, therefore
\begin{align*}
(\lambda (k_2) f) (g) = \int_{h \in K_{2}} \frac{1}{\mu (K_2)} \lambda (h) f(g) d \mu (h) = \\ \int_{h \in K_{2}} \frac{1}{\mu (K_2)} f(h^{-1} g) d \mu (h) = \int_{h \in K_{2}} \frac{1}{\mu (K_2)} f(h^{-1} h'g') d \mu (h) = \\ \int_{h \in K_{2}} \frac{1}{\mu (K_2)} f(h^{-1} g') d \mu (h) =(\lambda (k_2) f) (g').  
\end{align*}
This yields that for every $f \in L^2 (K_{1,2},\mu)$, $\lambda (k_2) f$ is fixed on right cosets of $K_2$ as we claimed. 

Next, denote $l_2 = [K_2 : K_{1} \cap K_{2}]$ and fix $h_1^2,...,h^2_{l_2}$ such that $K_2 = \bigcup_j (K_{1} \cap K_{2}) h_j^2$.  Let $f \in L^2 (K_{1,2},\mu)_1$, then for every $g \in K_{1,2}$ the following holds:
\begin{align*}
(\lambda (k_2) f) (g) = \int_{h \in K_{2}} \frac{1}{\mu (K_2)} \lambda (h) f(g) d \mu (h) = \int_{h \in K_{2}} \frac{1}{\mu (K_2)} f(h^{-1} g) d \mu (h) = \\ \sum_{j=1}^{l_2} \int_{h \in K_1 \cap K_{2}} \frac{1}{\mu (K_2)} f(h^{-1} h^2_j g) d \mu (h).     
\end{align*}
Note that $f \in L^2 (K_{1,2},\mu)_1$ and therefore for every $h \in K_{1,2}$ and for every $j$, $f(h^{-1} h^2_j g) = f(h^2_j g)$. Therefore
\begin{align*}
\sum_{j=1}^{l_2} \int_{h \in K_1 \cap K_{2}} \frac{1}{\mu (K_2)} f(h^{-1} h^2_j g) d \mu (h) = \sum_{j=1}^{l_2} \frac{\mu (K_1 \cap K_2)}{\mu (K_2)} f(h^2_j g)=  \\ \sum_{j=1}^{l_2} \frac{1}{l_2} f(h^2_j g) = (F_2^{-1} M_2 F_1 f) (g),
\end{align*}
as needed. Similarly
$$\forall f \in L^2 (K_{1,2},\mu)_2, \lambda (k_1) f = F_1^{-1} M_1 F_2 f.$$
Therefore,  
$$\forall f \in L^2 (K_{1,2},\mu)_2, \lambda (k_2) \lambda (k_1) f = F_2^{-1} M_2 M_1 F_2 f.$$
To finish, we notice that
$$\lambda ((k_1 k_2 - k_{1,2})^* (k_1 k_2 - k_{1,2})) = \lambda (k_2 k_1) \lambda (k_2 -k_{1,2}).$$ 
This implies that the non-trivial eigenvalues of $\lambda ((k_1 k_2 - k_{1,2})^* (k_1 k_2 - k_{1,2}))$ are in $(Ker (\lambda (k_2 - k_{1,2})))^\perp = Im (\lambda (k_2 - k_{1,2})) = Im (\lambda (k_2)) \cap Im (I-\lambda (k_{1,2}))$. By the definition of $k_{1,2}$ and $k_2$ we have that $Im (\lambda (k_2)) \cap Im (I-\lambda (k_{1,2}))$ is 
$$Im (\lambda (k_2)) \cap Im (I-\lambda (k_{1,2})) = \lbrace f \in L^2 (K_{1,2},\mu)_2 : \int_{K_{1,2}} f =0 \rbrace.$$
Therefore $Im (\lambda (k_2)) \cap Im (I-\lambda (k_{1,2}))$ is identified by $F_2$ with the space
$$\lbrace \psi \in L^2 (V)_2 : \sum_{v \in V_2} \psi(v) =0 \rbrace = span \lbrace \chi_{V_2} \psi_2,...,\chi_{V_2} \psi_{\vert V_2 \vert} \rbrace,$$
where the last equality is due to proposition \ref{bipartite graph spectrum proposition 3}. Therefore the non-trivial spectrum of $\lambda ((k_1 k_2 - k_{1,2})^* (k_1 k_2 - k_{1,2}))$ is the same as the non trivial spectrum of $M_2 M_1$ on $span \lbrace \chi_{V_2} \psi_2,...,\chi_{V_2} \psi_{\vert V_2 \vert} \rbrace$ as needed.
\end{proof}

\begin{corollary}
\label{Laplacian eigenvalues corollary}
Let $K_{1,2}$, $K_1, K_2$ be as above and let $\lambda$ be the left regular representation on $L^2 (K_{1,2},\mu)$, where $\mu$ is the Haar measure of $K_{1,2}$. Let $\Delta$ be the graph Laplacian of $\mathcal{G}$ defined above. Let $0 = \eta_1 < \eta_2 \leq \eta_3 \leq ...\leq` \eta_{\vert V \vert -1} < \eta_{\vert V \vert} =2$ be the eigenvalues (including multiplicities) of $\Delta$. 
Then for $k_1,k_2, k_{1,2} \in C_c (K_{1,2})$ defined as in the previous section we have that 
$$\Vert \lambda (k_1 k_2 - k_{1,2}) \Vert \leq 1- \eta_2,$$
and for every $1 \leq r < \infty$
\begin{align*}
\Vert \lambda (k_1 k_2 - k_{1,2}) \Vert_{S^r} \leq \left( ( 1-\eta_2 )^r +  (1-\eta_3 )^r +...+ ( 1-\eta_{\min \lbrace \vert V_1 \vert, \vert V_2 \vert \rbrace} )^r \right)^{\frac{1}{r}} \leq \\ (1-\eta_2)\min \lbrace \vert V_1 \vert^{\frac{1}{r}},\vert V_2 \vert^{\frac{1}{r}} \rbrace  .
\end{align*}
\end{corollary}

\section{Examples and applications}
\subsection{Groups acting on simplicial complexes}
We'll start by recalling some basic definitions regarding simplicial complexes. Let $\Sigma$ be a purely $n$-dimensional simplicial complex (i.e., every simplex is a face of an $n$-simplex). $\Sigma$ is called gallery connected if for every two $n$-dimensional simplices $\sigma, \sigma'$, there is a finite sequence of $n$-dimensional simplices $\sigma = \sigma_1, \sigma_2,...,\sigma_m = \sigma'$ such that for every $i$, $\sigma_i$ and $\sigma_{i+1}$ share an $(n-1)$-dimensional face. 

Recall that for every simplex $\sigma$ in $\Sigma$ one can define a new simplicial complex $link (\sigma)$ as the sub-complex of $\Sigma$ that contains all the simplices $\sigma '$ that are disjoint from $\sigma$ such that there is an $n$-dimensional simplex that contains both $\sigma$ and $\sigma '$. Observe that the dimension of $link (\sigma)$ is always $n - dim (\sigma)-1$. In particular, the $1$-dimensional links of $\Sigma$ are the links of simplices of dimension $n-2$. 

Assume that $\Sigma$ is a pure $n$-dimensional simplicial complex that is gallery connected  such that the $1$-dimensional links of $\Sigma$ are finite connected graphs. Assume further that $G$ is a group acting on $\Sigma$ simplicially such that the fundamental domain $\Sigma / G$ is a single $n$ dimensional simplex and such that the stabilizers of all the $(n-2)$-simplices of $\Sigma$ are compact. Fix $\lbrace v_1,...,v_{n+1} \rbrace \in \Sigma^{(n)}$ and denote $K_i$ to be the stabilizer of $\lbrace v_1,..., \hat{v_i},...,v_{n+1} \rbrace$ and $K_{i,j}$ to be the stabilizer of $\lbrace v_1,..., \hat{v_i},...,\hat{v_j},...,v_{n+1} \rbrace$. The assumption that $\Sigma$ is galley connected implies that $K_1,...,K_{n+1}$ generate $G$. Also, the assumptions on the $1$-dimensional links of $\Sigma$ imply that $K_{1,2},...,K_{n,n+1}$ are compact groups, $K_i,K_j$ are finite index subgroups of $K_{i,j}$ and $K_{i,j} = \langle K_i, K_j \rangle$. Further more, the $1$-dimensional link of $\lbrace v_1,..., \hat{v_i},...,\hat{v_j},...,v_{n+1} \rbrace$ can be identified with the graph $\mathcal{G}$ defined by $K_i, K_{j}$ and $K_{i,j}$ (see definition \ref{graph of K_1, K_2 definition} above).   

Using corollary \ref{Laplacian eigenvalues corollary} above, we can state the following theorem generalizing theorem \ref{groups acting on building theorem intro} stated in the introduction:
\begin{theorem}
Let $\Sigma$ be a pure $n$-dimensional simplicial complex that is galley connected. Let $G$ be a group acting simplicially on a $\Sigma$ such that the action is cocompact and the fundamental domain $\Sigma / G$ is a single $n$-dimensional simplex $\lbrace v_1,...,v_{n+1} \rbrace$ and such that the stabilizer of every $(n-2)$-dimensional simplex of $\Sigma$ is a compact subgroup of $G$. Assume that for every $\tau_{i,j} = \lbrace v_1,..., \hat{v_i},...,\hat{v_j},...,v_{n+1} \rbrace$, the link of $\tau_{i,j}$ is a finite connected (bipartite) graph $(V^{i,j},E^{i,j})$, where $V^{i,j}_1,V^{i,j}_2$ are the two sides of this graph. Denote
$$L = \max_{1 \leq i < j \leq n+1} \min \lbrace \vert V^{i,j}_1 \vert, \vert V^{i,j}_2 \vert \rbrace.$$ 

Assume further that there is a constant $\eta > 1- \frac{1}{8(n+1)-11}$ such that for every $1 \leq i < j \leq n +1$, the smallest positive eigenvalue of the Laplacian on the link of $\tau_{i,j}$ is $\geq \eta$. Fix a constant $c'$ such that $1-\eta \leq c' < \frac{1}{8(n+1)-11}$. Define the following Banach classes:
\begin{enumerate}
\item The class $\mathcal{E}_1  = \overline{B_{BM} (\mathcal{H}, \delta) \cup Int(\mathcal{H}, \geq \theta)},$
where $\delta = \frac{c'}{1-\eta} -1, \theta = \frac{\ln (2)- \ln (c')}{\ln (2) - \ln (1-\eta) }$, $\mathcal{H}$ is the class of all Hilbert spaces.
\item For $(p_1,p_2,r) \in  (1,2] \times [2, \infty) \times [2, \infty)$ such that $\frac{1}{p_1} - \frac{1}{p_2} < \frac{1}{r}$, the class $\mathcal{E} (p_1,p_2,r)$ defined as follows: 
$$ \overline{\bigcup_{\begin{scriptsize}
\begin{array}{c} 
T_{p_1}, C_{p_2} \in [1, \infty) \\
c(p_1,p_2,r,q) T_{p_1} C_{p_2} \leq c'
\end{array}  \end{scriptsize}} Int \left(\mathcal{T} (p_1, p_2, T_{p_1}, C_{p_2} ), \geq \frac{\ln (2) - \ln (c')}{\ln (2) - \ln (c(p_1,p_2,r,q) T_{p_1} C_{p_2})} \right)},$$
where $c(p_1,p_2,r,q) =L^{\frac{1}{r}} (1-\eta) \sum_i 2^{\frac{r}{r-1} (\frac{1}{p_1} - \frac{1}{p_2} - \frac{1}{r} )i}$.
Note that $\mathcal{E} (p_1,p_2,r)$ is non empty only if $c(p_1,p_2,r,q) \leq c'$.
\item The class $\mathcal{E}$ defined as follows: denote
$$A (q) = \left \lbrace (q_1,q_2, r) \in (1,2] \times [2, \infty) \times [2, \infty) :  \frac{1}{q_1} - \frac{1}{q_2} < \frac{1}{r}, c(q_1,q_2,r)  \leq c' \right\rbrace. $$
Then $\mathcal{E} = \overline{\mathcal{E}_1 \cup \bigcup_{(q_1,q_2,r) \in A(q)} \mathcal{E} (q_1,q_2,r)}.$
\end{enumerate}
Then $G$ has robust Banach property (T) with respect to $\mathcal{E}$ and has property $F_X$ for every $X \in \mathcal{E}$. In particular, $G$ has property $F_{L^p}$ for any $p < 2\frac{\ln (2) - \ln (1-\eta)}{\ln (2) + \ln (8(n+1)-11)}$. 
\end{theorem}

\begin{proof}
As noted above, by the assumption of the theorem, $G$ is generated by $K_1,...,K_{n+1}$ and for every $1 \leq i < j \leq n+1$ the link of $\lbrace v_1,..., \hat{v_i},...,\hat{v_j},...,v_{n+1} \rbrace$ can be identified with the graph $\mathcal{G}$ generated by $K_i, K_{j}, K_{i,j}$ as in definition \ref{graph of K_1, K_2 definition}. Therefore the proof of the theorem follows by applying corollary \ref{Laplacian eigenvalues lemma} combined with theorems \ref{mSBT from angles in unitary rep theorem}, \ref{mSBT via Schatten norms}. 
\end{proof}

\begin{remark}
The reader should Note the asymptotic behaviour of the above theorem as $\eta \rightarrow 1$. For instance, as $\eta \rightarrow 1$, $G$ has property $F_{L^p}$ with $p \rightarrow \infty$.
\end{remark}

\subsection{Groups where the $K_{i,j}$'s are Heisenberg or Abelian}
%Let $G$ be a discrete group generated by finite groups $K_1,...,K_N$ such that for each $1 \leq i < j \leq N$, $K_{i,j} = \langle K_i, K_j \rangle$ is a finite group. Let $r \geq 1$. There are two cases where calculating the $r$ Schatten norm for $k_i k_j - k_{i,j}$, $k_j k_i- k_{i,j}$ is relatively easy. The first case, is the trivial case when $K_i$ and $K_j$ commute. In this case, we have that $k_i k_j - k_{i,j} = k_j k_i- k_{i,j} = 0$ and the Schatten norm is $0$ for every $r$. The second case is if the following conditions are satisfied: 
%\begin{itemize}
%\item There is a prime $q$ such that both $K_i$ and $K_j$ are isomorphic to the field $\mathbb{F}_q$. Therefore we can identify  $K_i$ with the group whose elements are  $\lbrace x_i (s) : s \in \mathbb{F}_q \rbrace$ under the relation $x_i (s_1) x_i (s_2) = x_i (s_1+s_2)$ for every $s_1,s_2 \in  \mathbb{F}_q$ and   identify $K_j$ with $\lbrace x_j (s) : s \in \mathbb{F}_q \rbrace$ under a similar relation.
%\item Under the previous identification, we have for every $s_1, s_2 \in \mathbb{F}_q$ that $[x_i (s_1), x_j (s_2) ] = [x_i (1), x_j (s_1 s_2) ] $ and $[K_i,K_j]$ commute with both $
%\end{itemize}
 %and if
%$K_{i,j}$ is the $q$-He  there is a prime $q$ such that $K_i = K_j = \mathbb{F}_q$ and if we denote $K_i = 

Let $K=K_{1,2}$ be a finite group generated by two subgroups $K_1,K_2$ and let $k_1, k_2, k_{1,2}$ be the functions in $C(K)$ defined in the beginning of the previous section. As we saw in theorems \ref{mSBT criterion via regular rep}, \ref{mSBT via Schatten norms} and corollaries \ref{fixed point corollary from angles in unitary rep}, \ref{fixed point property via Schatten norms}, it is useful to be able to calculate the eigenvalues of $k_1 k_2- k_{1,2}$ in all the irreducible representations of $K$. There are two cases where calculating those is relatively easy. The first case is $K_{1,2}$ is Abelian (or more generally, when $K_1$ and $K_2$ commute). In this case we have that $k_1 k_2 - k_{1,2} = k_2 k_1- k_{1,2} = 0$ (all the eigenvalues are $0$). The second case is where $K$ is the Heisenberg group modulo $q$, $H_q$, for some prime $q$ and $K_1$, $K_2$ are the groups generated by the standard generator of $H_q$, i.e., 
$$K_1 = \langle x : x^q =1 \rangle,$$
$$K_2 = \langle y : y^q =1 \rangle,$$
$$K = \langle x,y :[[x,y],x] =[[x,y],y] = x^q = y^q = [x,y]^q =1 \rangle.$$
In this case, the irreducible representations of $K$ are well known and  therefore we can use them to bound $\cos (\angle (H^{\pi (K_1)}, H^{\pi (K_2)}))$ and calculate the $r$-Schatten norm for $k_1 k_2 - k_{1,2}, k_2 k_1 - k_{1,2}$ in $L^2 (K)$:
\begin{enumerate}
\item $K$ has $q^2$ irreducible representations of degree $1$. Note that for any representation $\pi$ of degree $1$, we always have $\pi (k_1 k_2 - k_{1,2}) =0$.
\item $K$ has $q-1$ irreducible representations of degree $q$ described as follows: every non trivial $q$-root of unity $\zeta$,  define the representation $\pi_\zeta$ on $\mathbb{C}^q$ as follows: let $e_1,...,e_q$ be the standard basis of $\mathbb{C}^q$, then $\pi_\zeta$ is defined as follows
$$\pi_\zeta (x).e_i =  \zeta^i e_i , \pi_\zeta (y).e_i = e_{i+1} .$$
One can calculate that $\pi_\zeta (k_1 k_2 - k_{1,2})$ has $\frac{1}{\sqrt{q}}$ as an eigenvalue of multiplicity $1$ and all the other eigenvalues are $0$.
\item From the above, we get that for every unitary representation $\pi$ of $K$ on a Hilbert space $H$ we have that
$$\cos (\angle (H^{\pi (K_1)}, H^{\pi (K_2)})) \leq \frac{1}{\sqrt{q}}.$$ 
Also, using Peter-Weyl theorem to decompose $L^2 (K)$ as matrix coefficients, we can find that  
$$\Vert \lambda (k_1 k_2 - k_{1,2}) \Vert_{S^r} = \dfrac{(q^2-q)^{\frac{1}{r}}}{\sqrt{q}}.$$
\end{enumerate}    
Using the above computation, we can prove the following theorem:
\begin{theorem}
\label{Abelian or Heisenberg K_ij theorem}
Let $G$ be a discrete group generated by finite Abelian subgroups $K_1,...,K_N$ of order $q$, where $q$ is prime such that $q > (8N-11)^2$. Assume that for every $1 \leq i < j \leq N$ one of the following holds: either $K_i$ and $K_j$ commute (and therefore $K_{i,j}$ is $\mathbb{F}_q \times \mathbb{F}_q$) or $K_{i,j} = H_q$  and $K_i$, $K_j$ are the subgroups generated by the standard generators of $H_q$. Fix a constant $c'$ such that $\frac{1}{\sqrt{q}} \leq c' < \frac{1}{8N-11}$. Define the following Banach classes:
\begin{enumerate}
\item The class $\mathcal{E}_1  = \overline{B_{BM} (\mathcal{H}, \delta) \cup Int(\mathcal{H}, \geq \theta)},$
where $\delta = c' \sqrt{q}  -1, \theta = \frac{\ln (2)- \ln (c')}{\ln (2) + \ln (\sqrt{q}) }$, $\mathcal{H}$ is the class of all Hilbert spaces.
\item For $(p_1,p_2,r) \in  (1,2] \times [2, \infty) \times [2, \infty)$ such that $\frac{1}{p_1} - \frac{1}{p_2} < \frac{1}{r}$, the class $\mathcal{E} (p_1,p_2,r)$ defined as follows: 
$$ \overline{\bigcup_{\begin{scriptsize}
\begin{array}{c} 
T_{p_1}, C_{p_2} \in [1, \infty) \\
c(p_1,p_2,r,q) T_{p_1} C_{p_2} \leq c'
\end{array}  \end{scriptsize}} Int \left(\mathcal{T} (p_1, p_2, T_{p_1}, C_{p_2} ), \geq \frac{\ln (2) - \ln (c')}{\ln (2) - \ln (c(p_1,p_2,r,q) T_{p_1} C_{p_2})} \right)},$$
where $c(p_1,p_2,r,q) =\frac{(q^2 - q)^{\frac{1}{r}}}{\sqrt{q}} \sum_i 2^{\frac{r}{r-1} (\frac{1}{p_1} - \frac{1}{p_2} - \frac{1}{r} )i}$.
Note that $\mathcal{E} (p_1,p_2,r)$ is non empty only if $c(p_1,p_2,r,q) \leq c'$.
\item The class $\mathcal{E}$ defined as follows: denote
$$A (q) = \left \lbrace (q_1,q_2, r) \in (1,2] \times [2, \infty) \times [2, \infty) :  \frac{1}{q_1} - \frac{1}{q_2} < \frac{1}{r}, c(q_1,q_2,r)  \leq c' \right\rbrace. $$
Then $\mathcal{E} = \overline{\mathcal{E}_1 \cup \bigcup_{(q_1,q_2,r) \in A(q)} \mathcal{E} (q_1,q_2,r)}.$
\end{enumerate}
Then $G$ has robust Banach property (T) with respect to $\mathcal{E}$ and has property $F_X$ for every $X \in \mathcal{E}$. In particular, $G$ has property $F_{L^p}$ for any $p < 2\frac{\ln (2) + \ln (\sqrt{q})}{\ln (2) + \ln (8N-11)}$. 
\end{theorem}

\begin{proof}
Using the above computations for Heisenberg groups $H_q$, we have that
\begin{itemize}
\item For every $1 \leq i < j \leq N$ and every unitary representation $\pi$ of $K$ on a Hilbert space $H$ we have that
$$\cos (\angle (H^{\pi (K_i)}, H^{\pi (K_j)})) \leq \frac{1}{\sqrt{q}}.$$ 
\item For every $r \geq 2$ we have that $\cos^{S^r}_\max \leq \frac{(q^2 -q)^{\frac{1}{r}}}{\sqrt{q}}$. 
\end{itemize}
Therefore we can apply theorems \ref{mSBT from angles in unitary rep theorem}, \ref{mSBT via Schatten norms} and proposition \ref{fixed point proposition} to get the above theorem. 
\end{proof}

\begin{remark}
\label{1/p1-1/p2 < 1/4 remark}
The reader should note the asymptotic behaviour of the above theorem when $q \rightarrow \infty$. For instance, for a group $G$ as above and any constants $(p_1,p_2) \in  (1,2] \times [2, \infty)$, $T_{p_1} \geq 1, C_{p_2} \geq 1$, $\theta >0$, if $\frac{1}{p_1} - \frac{1}{p_2} < \frac{1}{4}$, then there is a constant $q(p_1,p_2, T_{p_1}, C_{p_2}, \theta,N)$ such that for any prime $q \geq q(p_1,p_2, T_{p_1}, C_{p_2}, \theta,N)$, we have that for $\mathcal{E}$ defined above 
$$\overline{Int (\mathcal{T} (p_1,p_2, T_{p_1}, C_{p_2}), \geq \theta)} \subset \mathcal{E}.$$
Below we shall use this asymptotic behaviour to construct an expander family of graphs that does not uniformly coarsely embed in any Banach space $X \in \overline{Int (\mathcal{T} (p_1,p_2, T_{p_1}, C_{p_2}), \geq \theta)}$ for given $(p_1,p_2) \in  (1,2] \times [2, \infty)$, $T_{p_1} \geq 1, C_{p_2} \geq 1$, $\theta >0$ as above.
\end{remark}
Below we shall give some examples of groups for which the conditions of theorem \ref{Abelian or Heisenberg K_ij theorem} are fulfilled.

\subsubsection{Kac-Moody-Steinberg groups over a finite field}
Basic Kac-Moody-Steinberg groups where introduced in \cite{ErshovJZ} as follows:
\begin{definition}
Let $(V,E)$ be a finite graph without loops or multiple edges and let $R$ be a ring. For convenience, we denote  $V =  \lbrace 1,...,N \rbrace$. Define the group $G=G((V,E),R)$ as follows. First, define the groups $K_i$ for every $i \in V$ as the groups with elements $\lbrace x_i (s) : s \in R \rbrace$ under the relations $x_i (s_1) x_i (s_2) = x_i (s_1+s_2)$ for every $s_1,s_2 \in R$. Next define the group $G$ generated by $K_1,...,K_N$ under the following relations:
\begin{itemize}
\item For $1 \leq i < j \leq N$, if $\lbrace i, j \rbrace \notin E$, then $K_i$ and $K_j$ commute.
\item For $1 \leq i < j \leq N$, if $\lbrace i, j \rbrace \in E$, then for every $s_1,s_2 \in R$ we have that $[x_i (s_1), x_j (s_2) ] = [x_1 (1), x_j (s_1 s_2)]$ and $[K_i,K_j]$ commutes with both $K_i$ and $K_j$.
\end{itemize}
\end{definition}
We note that when $R = \mathbb{F}_q$, then for $G=G((V,E),\mathbb{F}_q)$ the following holds:
\begin{itemize}
\item Every $K_1,...,K_N$ are isomorphic to $\mathbb{F}_q$.
\item For $1 \leq i < j \leq N$, if $\lbrace i, j \rbrace \notin E$, then $K_i$ and $K_j$ commute.
\item For $1 \leq i < j \leq N$, if $\lbrace i, j \rbrace \in E$, then $K_{i,j}$ is the $q$ Heisenberg group $H_q$ and $K_i$, $K_j$ are the subgroups generated by the standard generators of $H_q$.
\end{itemize} 
Therefore for any graph $(V,E)$, $G((V,E),\mathbb{F}_q)$ fulfils the conditions of theorem \ref{Abelian or Heisenberg K_ij theorem} above.

\subsubsection{The groups $St_n (\mathbb{F}_q [ t_1,...,t_m ])$ and $EL_n ( \mathbb{F}_q [ t_1,...,t_m ])$}
For a ring $R$ and for $n \geq 3$, the Steinberg group $St_n (R)$ is a group generated by elements $x_{i,j} (s)$ where $1 \leq i \neq j \leq n$ and $s \in R$ which has the following defining relations:
$$\forall 1 \leq i \neq j \leq n, \forall s_1,s_2 \in R, x_{i,j} (s_1) x_{i,j} (s_2) = x_{i,j} (s_1+s_2),$$
$$\forall 1 \leq i \neq l \leq n, \forall s_1,s_2 \in R, [x_{i,j} (s_1), x_{k,l} (s_2)] = \begin{cases}
1 &  j \neq k \\
x_{i,l} (s_1 s_2) &  j=k
\end{cases}
.$$
The group of elementary matrices over $R$, denoted $EL_n (R)$ is the group of $n \times n$ matrices with entries in $R$, generated by the matrices $e_{i,j} (s)$ for $1 \leq i \neq j \leq n$ and $s \in R$, where $e_{i,j} (s)$ denotes the elementary matrix with $1$ on the diagonal, $s$ in the $(i,j)$ entry and $0$ in all the other entries. One can easily check that the relations specified above for $St_n (R)$ also hold for $EL_n (R)$, i.e., that
$$\forall 1 \leq i \neq j \leq n, \forall s_1,s_2 \in R, e_{i,j} (s_1) e_{i,j} (s_2) = e_{i,j} (s_1+s_2),$$
$$\forall 1 \leq i \neq l \leq n, \forall s_1,s_2 \in R, [e_{i,j} (s_1), e_{k,l} (s_2)] = \begin{cases}
1 &  j \neq k \\
e_{i,l} (s_1 s_2) &  j=k
\end{cases}
.$$
We shall show that for the ring of polynomials $\mathbb{F}_q [t_1,...,t_m]$, $St_n (\mathbb{F}_q [ t_1,...,t_m ])$ and $EL_n ( \mathbb{F}_q [ t_1,...,t_m ])$ fulfil the conditions of theorem \ref{Abelian or Heisenberg K_ij theorem}. We will show this only for $St_n (\mathbb{F}_q [ t_1,...,t_m ])$ since the proof is identical for both groups. Denote $t_0 =1 \in \mathbb{F}_q [t_1,...,t_m]$ and define the following subgroups $K_1	,...,K_{n+m}$ of $St_n (\mathbb{F}_q [ t_1,...,t_m ])$: 
$$\forall 1 \leq i  < n, K_i = \lbrace x_{i,i+1} (a \cdot 1) : a \in \mathbb{F}_p \rbrace,$$
$$\forall n \leq i \leq n+m, K_i = \lbrace x_{n,1} (a \cdot t_{i-n}) : a \in \mathbb{F}_p \rbrace.$$
To see that $K_1,...,K_{n+m}$ fulfil the conditions of theorem \ref{Abelian or Heisenberg K_ij theorem}, note the following:
\begin{itemize}
\item All the $K_i$'s are isomorphic to the group $\mathbb{F}_q$.
\item $K_1,..,K_{n+m}$ generate $St_n (\mathbb{F}_q [ t_1,...,t_m ])$.
\item For any $1 \leq i < j \leq n+m$, we have the following:
$$K_{i,j} = \begin{cases}
\mathbb{F}_q \times \mathbb{F}_q & 1 \leq i < j < n, j - i >1 \\
\mathbb{F}_q \times \mathbb{F}_q & 1 < i < n-1, n \leq j \leq n+m \\
\mathbb{F}_q \times \mathbb{F}_q & n \leq i < j \leq n +m \\
H_q & 1 \leq i <n-1, j = i+1 \\
H_q & i =1, n \leq j \leq n+m \\
H_q & i =n-1, n \leq j \leq n+m \\
\end{cases}.$$
\end{itemize}

Applying theorem \ref{Abelian or Heisenberg K_ij theorem} gives a result which generalize theorems \ref{Steinberg group theorem intro}, \ref{Steinberg group theorem intro2} stated in the introduction. 

\begin{remark}
In the case that $n \geq 4$, Mimura \cite{Mimura} using a completely different approach showed that for any finitely generated, unital, and associative ring $R$, $EL_n (R)$ and $St_n (R)$ have fixed point properties for every $L^p$ space provided that $p \in [1,\infty)$ and every non commutative $L^p$ space provided that $p \in (1, \infty)$ (see \cite{Mimura}[Corollary 1.4]). Therefore for fixed point properties for $L^p$ spaces (and non commutative $L^p$ spaces) Mimura's results are stronger than ours. However, one should note that our results covers fixed point properties for Banach spaces that are not superreflexive, which are not achieved by Mimura's work. We also deal with the case $n=3$, which is also not covered by Mimura's results. 
\end{remark}

\begin{remark}
We chose to phrase our results of the ring $\mathbb{F}_q [t_1,...,t_m]$, but the same proof the we gave above will also work for the ring $\mathbb{F}_q \langle t_1,...,t_m \rangle$.
\end{remark}

\subsection{Construction of Banach expanders}
Here we shall use our results regarding robust Banach property (T) of $EL_n (\mathbb{F}_q ([t_1,...,t_m])$ proven above to construct a family of graphs of uniformly bounded valency that are Banach expanders with respect to a large class of Banach spaces generalizing theorem \ref{Expander construction intro} given in the introduction. To be specific, for every constants $p_1 \in (1,2], p_2 \in [2, \infty), T_{p_1} \geq 1, C_{p_2} \geq 2, \theta >0$, such that $\frac{1}{p_1} - \frac{1}{p_2} < \frac{1}{4}$, we will construct a sequence of expanders that does not uniformly coarsely embed in any $X \in \overline{Int( \mathcal{T} (p_1,p_2,T_{p_1},C_{p_2}), \geq \theta)}$.
%$$\mathcal{E} = \bigcup_{\begin{scriptsize}
%\begin{array}{c}
%p_1 \in (1,2], p_2 \in [2, \infty)\\
 %\frac{1}{p_1} - \frac{1}{p_2} < \frac{1}{4}
%\end{array}
%\end{scriptsize}}  \bigcup_{T_{p_1} \geq 1, C_{p_2} \geq 1} \bigcup_{\theta >0} \overline{Int( \mathcal{T} (p_1,p_2,T_{p_1},C_{p_2}), \geq \theta) } .$$

Let $n \geq 3, m \geq 1$ and let $q$ be some prime. We showed above that $EL_n (\mathbb{F}_q [t_1,...,t_m])$ fulfils the conditions of theorem \ref{Abelian or Heisenberg K_ij theorem}. Therefore by remark \ref{1/p1-1/p2 < 1/4 remark}, for any constants $p_1 \in (1,2], p_2 \in [2, \infty), T_{p_1} \geq 1, C_{p_2} \geq 1, \theta >0$, such that $\frac{1}{p_1} - \frac{1}{p_2} < \frac{1}{4}$, $EL_n (\mathbb{F}_q [t_1,...,t_m])$ has robust Banach property (T) with respect to $\overline{Int( \mathcal{T} (p_1,p_2,T_{p_1},C_{p_2}), \geq \theta)}$ provided that $q$ is large enough. 

Therefore by proposition \ref{expander proposition}, in order to construct a sequence of expanders with respect to $\overline{Int( \mathcal{T} (p_1,p_2,T_{p_1},C_{p_2}), \geq \theta)}$ it is enough to find a sequence of normal finite index groups $N_i < EL_n (\mathbb{F}_q [t_1,...,t_m])$ such that $\bigcap N_i = \lbrace 1 \rbrace$. This can be achieved by considering principal congruence subgroups of $EL_n (\mathbb{F}_q [t_1,...,t_m])$: for every $i \in \mathbb{N}$, denote $I_i$ to be the two sided ideal of $\mathbb{F}_q [t_1,...,t_m]$ that is generated by all the monomials in $t_1,...,t_m$ of degree $i$. Let $\psi_i$ be the homomorphism:
$$\psi_i :  EL_n (\mathbb{F}_q [t_1,...,t_m]) \rightarrow  EL_n (\mathbb{F}_q [t_1,...,t_m] / I_i),$$
and let $N_i = ker (\psi_i)$. From the fact that $\vert \mathbb{F}_q [t_1,...,t_m] / I_i \vert < \infty$, we get that $N_i$ is always of finite index and one can easily see that $\bigcap_i N_i = \lbrace 1 \rbrace$. Therefore we can conclude this discussion by stating our result:

\begin{proposition}
Let $p_1 \in (1,2], p_2 \in [2, \infty), T_{p_1} \geq 1, C_{p_2} \geq 1, \theta >0$ be constants such that $\frac{1}{p_1} - \frac{1}{p_2} < \frac{1}{4}$. For any $n \geq 3$ and any $m \geq 1$, there is a large enough prime $q$ such that for any fixed symmetric generating set $S$ of $EL_n (\mathbb{F}_q [t_1,...,t_m])$, we have that the Cayley graphs  $\lbrace (EL_n (\mathbb{F}_q [t_1,...,t_m]) / N_i, S) \rbrace_{i \in \mathbb{N}}$ is a family of $X$-expanders for any $X \in \overline{Int( \mathcal{T} (p_1,p_2,T_{p_1},C_{p_2}), \geq \theta)}$.
\end{proposition}

\appendix 
\section{Applications of robust Banach property (T)}
In this appendix, we'll prove the applications of robust Banach property (T) for fixed point properties and for Banach expanders. In both cases the proofs are just minor adaptations of the proofs of  Lafforgue in \cite{Lafforgue2}.
\subsection{Fixed point property application}
We shall prove the following:
\begin{proposition}
Let $X$ be a Banach space and let $G$ be a locally compact group. If $G$ has robust property (T) with respect to $\mathbb{C} \oplus X$ with the $l_2$ norm, then any affine isometric action of $G$ on $X$ has a fixed point. 
\end{proposition}

\begin{proof}
Let $\rho$ be an isometric action of $G$ on $X$. Let $\overline{0} \in X$ be the zero of (the underlying vector space of) $X$ and define a length $l$ over $G$ as 
$$l(g) = \max \lbrace \Vert \rho (g).\overline{0}  \Vert, 1 \rbrace.$$
$G$ has robust property (T) with respect to $\mathbb{C} \oplus X$ and therefore there is some $s_0 >0$ and a sequence of positive symmetric real functions $f_n \in C_c (G)$ with $\int f_n =1$ such for every representation $\pi$ of $G$ on $\mathbb{C} \oplus X$ and for any $0 \leq s \leq s_0$, if $\Vert \pi (g) \Vert \leq e^{s l (g)}$, then $\pi (f_n)$ converges to $\pi (p)$ that is a projection on  $(\mathbb{C} \oplus X )^\pi$. \\
Fix $D>1$ to be a constant whose value will be determined later and define a representation $\pi$ as on $\mathbb{C} \oplus X$ as follows: $\pi$ is the unique representation on $ \mathbb{C} \oplus X$ such that 
$$\forall g \in G, \forall v \in X, \pi (g). (D,v) = (D,\rho (g).v).$$
In other words, $\pi$ is the representation that keeps $ D \oplus X$ invariant and acts on it via $\rho$. Next, we'll show that
$$\Vert \pi (g) \Vert \leq 1 + \sqrt{\dfrac{3}{D}} l(g), \forall g \in G.$$
Indeed,
\begin{dmath}
\label{bounding norm of pi - appendix}
\dfrac{\Vert \pi (g).(D,v) \Vert^2}{\Vert (D,v) \Vert^2} = \dfrac{D^2 + \Vert \rho (g). v \Vert^2}{D^2 + \Vert v \Vert^2} \leq  \dfrac{D^2 + (\Vert v \Vert+l(g) )^2}{D^2 + \Vert v \Vert^2}  = {1 + \dfrac{2 \Vert v \Vert l(g) + l(g)^2}{D^2 + \Vert v \Vert^2}} \leq \\ {1 + \dfrac{2 \Vert v \Vert +1}{D^2 + \Vert v \Vert^2} l(g)^2}.
\end{dmath}
Note that if $\Vert v \Vert \geq D$, then 
$$\dfrac{2 \Vert v \Vert +1}{D^2 + \Vert v \Vert^2} \leq \dfrac{3 \Vert v \Vert}{\Vert v \Vert^2} \leq \dfrac{3}{D}.$$
On the other hand, if $\Vert v \Vert < D$, then
$$\dfrac{2 \Vert v \Vert +1}{D^2 + \Vert v \Vert^2} \leq \dfrac{3 D}{D^2} = \dfrac{3}{D}.$$
Therefore, we have that for all $v$ that 
$$\dfrac{2 \Vert v \Vert +1}{D^2 + \Vert v \Vert^2} \leq  \dfrac{3}{D}.$$
Combined with \eqref{bounding norm of pi - appendix}, this yields
$$\dfrac{\Vert \pi (g).(D,v) \Vert^2}{\Vert (D,v) \Vert^2} \leq 1 + \dfrac{3}{D} l(g)^2,$$
and therefore
$$\dfrac{\Vert \pi (g).(D,v) \Vert}{\Vert (D,v) \Vert} \leq 1 + \sqrt{ \dfrac{3}{D} }l(g).$$
The above inequality implies that 
$$\Vert \pi (g) \Vert \leq 1 + \sqrt{\dfrac{3}{D}} l(g), \forall g \in G,$$
as needed. By choosing $D$ large enough, we can therefore insure that we'll have
$$\Vert \pi (g) \Vert \leq 1 + s_0 l(g) \leq e^{s_0 l(g)}, \forall g \in G.$$
Therefore $\pi$ meets the condition for robust Banach property (T) for $\mathbb{C} \oplus X$. Let $\lbrace f_n \rbrace$ be the sequence as in the definition of robust Banach property (T). Note that for every $n$ and every $v \in X$, $\pi (f_n). (D,v) \in D \oplus X$, since for every $n$, $\int f_n =1$. Fix some $v \in X$ and note that $\pi (p). (D,v)= \lim_n \pi (f_n).v \in D \oplus X$ and therefore there is some $v_0 \in X$ such that $\pi (p). (D,v) = (D,v_0)$. By the definition of $\pi$, $v_0$ is a fixed point of the action of $G$ on $X$ through $\rho$ and we are done.

\end{proof}

\subsection{Banach expanders application}
We shall prove the following:
\begin{proposition}
\label{expander proposition - appendix}
Let $G$ be a finitely generated discrete group and let $\lbrace N_i \rbrace_{i \in \mathbb{N}}$ be a sequence of finite index normal subgroups of $G$ such that $\bigcap_i N_i = \lbrace 1 \rbrace$. Let $\mathcal{E}$ be a class of Banach spaces that is closed under $l_2$ sums. Fix $S$ to be some symmetric generating set of $G$. If $G$ has robust Banach property (T) with respect to $\mathcal{E}$, then the family of Cayley graphs $\lbrace (G/N_i,S) \rbrace_{i \in \mathbb{N}}$ is a family of $X$-expanders for any $X \in \mathcal{E}$.    
\end{proposition}
We shall start by proving the following:
\begin{proposition}
\label{spectral gap type proposition}
Let $G$ be a discrete finitely generated group with a symmetric generating set $S$ and let $N_i$ be a sequence of finite index normal subgroups of $G$ such that $\bigcap_i N_i = \lbrace 1  \rbrace$. Denote by $(V_i,E_i)$ the Cayley graph of $G / N_i$ with respect to $S$. Let $\mathcal{E}$ be a class of Banach spaces such that $\mathcal{E}$ is closed under $l_2$ sums. Assume that $G$ has robust Banach property (T) with respect to $\mathcal{E}$, then there is a constant $C$ such that for every $X \in \mathcal{E}$, every $i$ and every map $\phi :(V_i, E_i) \rightarrow X$, we have that there is some $v (\phi) \in X$ such that
$$ \sum_{x \in V_i} \Vert \phi (x) - v (\phi) \Vert^2_X \leq C \sum_{\lbrace x, y \rbrace \in E_{i}} \Vert \phi (x) - \phi (y) \Vert^2_X.$$
\end{proposition}

\begin{proof}
Fix some $X \in \mathcal{E}$ and some $i$. Consider $L^2 (G / N_i ; X)$ with the representation $\pi : G \rightarrow L^2 (G / N_i ; X)$, defined as $\pi (g). \phi (g') = \phi (g' g)$. Then $L^2 (G / N_i ; X)$ is the $l_2$ sum of $[G:N_i]$ copies of $X$ and therefore $L^2 (G / N_i ; X) \in \mathcal{E}$. Note that $\pi$ is an isometric representation on $L^2 (G / N_i ; X)$ and therefore $\pi \in \mathcal{F} (\mathcal{E}, 0)$. 

From the fact that $G$ has robust Banach property (T) on $\mathcal{E}$, we get that there is $p \in C_{\mathcal{F} (\mathcal{E},0)}$ such that $\pi (p)$ is the projection on $ L^2 (G / N_i ; X)^\pi$. 
Note that the space of invariant vectors under $\pi$ is exactly the space of constant functions, so for every $\phi \in L^2 (G / N_i ; X)$, we can define $v (\phi ) \in X$ as the constant value of $\pi (p).\phi$. By the definition of robust Banach property (T) there is a real function $f \in C_c (G)$ such that $\int f =1$ and $\Vert p - f \Vert_{\mathcal{F} (\mathcal{E},0)} \leq \frac{1}{2}$. 

Note that for every $\phi$ we have that $\pi (f) \pi (p). \phi = \pi (p) . \phi$. Using this equality we get that $(\pi (f)-\pi (p)). \phi = (\pi (f)-\pi (p)).(\phi - \pi (p).\phi)$. Therefore
\begin{dmath*}
\Vert \phi - \pi (p).\phi \Vert_{L^2 (G/N_i ; X)}  \leq  \\
\Vert \phi - \pi (f).\phi \Vert_{L^2 (G/N_i ; X)} + \Vert \pi (f).\phi - \pi (p).\phi \Vert_{L^2 (G/N_i ; X)} =\\  \Vert \phi - \pi (f).\phi \Vert_{L^2 (G/N_i ; X)} + \Vert (\pi (f)-\pi (p)).(\phi - \pi (p).\phi) \Vert_{L^2 (G/N_i ; X)}  \leq  \\ \Vert \phi - \pi (f).\phi \Vert_{L^2 (G/N_i ; X)} + \frac{1}{2} \Vert \phi - \pi (p).\phi \Vert_{L^2 (G/N_i ; X)}.
\end{dmath*}
This yields that 
$$\Vert \phi - \pi (p).\phi \Vert_{L^2 (G/N_i ; X)}  \leq 2 \Vert \phi - \pi (f).\phi \Vert_{L^2 (G/N_i ; X)}$$
or equivalently
\begin{equation*}
\Vert \phi - \pi (p).\phi \Vert_{L^2 (G/N_i ; X)}^2  \leq 4 \Vert \phi - \pi (f).\phi \Vert_{L^2 (G/N_i ; X)}^2.
\end{equation*}
Note that by definition 
\begin{equation*}
\Vert \phi - \pi (p).\phi \Vert_{L^2 (G/N_i ; X)}^2  = \sum_{g \in G/N_i} \Vert \phi (g) - v (\phi) \Vert^2_X = \sum_{x \in V_i} \Vert \phi (x) - v (\phi) \Vert^2_X.
\end{equation*}
Therefore
\begin{equation}
\label{appendix expanders inequality}
\sum_{x \in V_i} \Vert \phi (x) - v (\phi) \Vert^2_X \leq 4 \Vert \phi - \pi (f).\phi \Vert_{L^2 (G/N_i ; X)}^2.
\end{equation}
Recall that $f$ is compactly supported and therefore there is some $k \in \mathbb{N}$ such that $supp (f) \subseteq S^k$. Note that for every $g \in S^k$, we have by the triangle inequality that 
$$\Vert \phi - \pi (g).\phi \Vert_{L^2 (G/N_i ; X)} \leq k \sum_{s \in S} \Vert \phi - \pi (s).\phi \Vert_{L^2 (G/N_i ; X)} .$$
Therefore if we denote $M = \max_{g \in G} \vert f(g) \vert$, we get that 
$$ \Vert \phi - \pi (f).\phi \Vert_{L^2 (G/N_i ; X)} \leq M k \sum_{s \in S} \Vert \phi - \pi (s).\phi \Vert_{L^2 (G/N_i ; X)} .$$
This yields that 
\begin{dmath*}
 \Vert \phi - \pi (f).\phi \Vert_{L^2 (G/N_i ; X)}^2 \leq  \left(M k  \sum_{s \in S} \Vert \phi - \pi (s).\phi \Vert_{L^2 (G/N_i ; X)} \right)^2 \leq \\
 M^2 k^2 \vert S \vert  \sum_{s \in S} \Vert \phi - \pi (s).\phi \Vert_{L^2 (G/N_i ; X)}^2 =  M^2 k^2 \vert S \vert \sum_{g \in G / N_i} \sum_{s \in S}  \Vert \phi (g) - \pi (s).\phi (g) \Vert_X^2 = \\
  M^2 k^2 \vert S \vert \sum_{g \in G / N_i} \sum_{s \in S}  \Vert \phi (g) - \phi (gs) \Vert_X^2 = M^2 k^2 \vert S \vert \sum_{x \in V_i} \sum_{y \in V_i, \lbrace x,y \rbrace \in E_i} \Vert \phi (x) - \phi (y) \Vert_X^2 =  \\
  2 M^2 k^2 \vert S \vert \sum_{\lbrace x, y \rbrace \in E_i } \Vert \phi (x) - \phi (y) \Vert_X^2.
\end{dmath*}
We are done by combining the above with \eqref{appendix expanders inequality}. Note that $M,k, \vert S \vert$ are all independent of the choice of $N_i$ and therefore taking $C = 8 M^2 k^2 \vert S \vert$ gives a constant that is uniform for all $N_i$'s.
\end{proof}

Using this proposition, we can prove proposition \ref{expander proposition - appendix}, by proving the following lemma:
\begin{lemma}
Let $X$ be a Banach space and let $\lbrace (V_i,E_i) \rbrace_{i \in \mathbb{N}}$ be a family of graphs with uniformly bounded valency such that $\vert V_i \vert \rightarrow \infty$ as $i \rightarrow \infty$. Assume that there is a constant $C$ such that for every $i \in \mathbb{N}$ and every $\phi : (V_i,E_i) \rightarrow X$ there is some $v (\phi) \in X$ such that 
$$ \sum_{x \in V_i} \Vert \phi (x) - v (\phi) \Vert^2_X \leq C \sum_{\lbrace x, y \rbrace \in E_{i}} \Vert \phi (x) - \phi (y) \Vert^2_X.$$
Then $\lbrace (V_i,E_i) \rbrace_{i \in \mathbb{N}}$ is a family of $X$-expanders. 
\end{lemma}

\begin{proof}
Let $D \geq 2$ be a uniform bound on the valency of $\lbrace (V_i,E_i) \rbrace_{i \in \mathbb{N}}$. Assume towards contradiction that there is a sequence of maps $\phi_i : V_i \rightarrow X$ and functions $\rho_-, \rho_+ : \mathbb{N} \rightarrow \mathbb{R}$ such that $\lim_k \rho_- (k) = \infty$ and 
$$\forall i \in \mathbb{N}, \forall x,y \in V_i, \rho_{-} (d_i (x,y)) \leq \Vert \phi_i (x) - \phi_i (y) \Vert \leq \rho_+ (d_i (x,y)),$$
where $d_i (x,y)$ is the distance in the graph $(V_i,E_i)$ between $x$ and $y$. By replacing $\phi_i$ by $\phi_i - v (\phi_i)$, we can assume that for every such $\phi_i$, we have that 
$$ \sum_{x \in V_i} \Vert \phi_i (x) \Vert^2_X \leq C \sum_{\lbrace x, y \rbrace \in E_{i}} \Vert \phi_i (x) - \phi_i (y) \Vert^2_X.$$
Note that
$$\sum_{\lbrace x, y \rbrace \in E_{i}} \Vert \phi_i (x) - \phi_i (y) \Vert^2_X \leq \vert E_i \vert \rho_+ (1)^2 \leq \dfrac{D \vert V_i \vert}{2} \rho_+ (1)^2.$$
Therefore
 $$\sum_{x \in V_i} \Vert \phi_i (x) \Vert^2_X \leq  \frac{\vert V_i \vert}{2} (C D \rho_+ (1)^2).$$
 Consider the median value of the multiset $\lbrace \Vert \phi_i (x) \Vert : x \in V_i \rbrace$. If this median is strictly greater than $\sqrt{C D} \rho_+ (1)$, we get a contradiction to the above inequality. Therefore, there is a set $U_i \subseteq V_i$ such that $\vert U_i \vert \geq \lfloor \frac{\vert V_i \vert}{2} \rfloor$ and
 $$\forall x \in U_i, \Vert \phi_i (x) \Vert \leq \sqrt{C D} \rho_+ (1).$$
 Therefore by triangle inequality
\begin{equation}
\label{appendix expanders ineq 3}
\forall x,y \in U_i, \Vert \phi_i (x) - \phi_i (y) \Vert \leq 2\sqrt{CD} \rho_+ (1).
\end{equation}

On the other hand, since the valency in all the graphs is bounded by $D$, we have that 
$$\forall i \in \mathbb{N}, \forall k \in \mathbb{N}, \forall x \in V_i, \vert \lbrace y \in V_i : d_i (x,y) < k \rbrace \vert < D^{k}.$$
Denote $diam (U_i)$ to be the diameter of $U_i$ in $V_i$, then by the above inequality we get that 
$$ diam (U_i) \geq \frac{\ln (\vert U_i \vert)}{\ln (D)} \geq \frac{\ln (\lfloor \frac{\vert V_i \vert}{2} \rfloor)}{\ln (D)}.$$
Therefore there are $x,y \in U_i$ such that 
$$\rho_{-} \left( \frac{\ln (\lfloor \frac{\vert V_i \vert}{2} \rfloor)}{\ln (D)} \right) \leq \Vert \phi_i (x) - \phi_i (y) \Vert.$$ 
Combining this with \eqref{appendix expanders ineq 3} yields that for every $i$ we have that 
$$\rho_{-} \left( \frac{\ln (\lfloor \frac{\vert V_i \vert}{2} \rfloor)}{\ln (D)} \right) \leq 2\sqrt{CD} \rho_+ (1).$$
But from the assumption that $\lim_i \vert V_i \vert = \infty$ we get a contradiction to the assumption that $\lim_k \rho_- (k) = \infty$.
\end{proof}

\bibliographystyle{plain}
\bibliography{bibl}

\end{document}